\documentclass{amsart}
\usepackage{amsopn}
\usepackage{amssymb, amscd, amsmath}
\usepackage{array}
\usepackage{color} 
\usepackage{enumerate}
\usepackage{tikz-cd}

\newcommand{\nc}{\newcommand}

\DeclareRobustCommand{\SkipTocEntry}[4]{}

\nc{\sff}{\operatorname{II}}
\nc{\QbHm}{\mathsf{Q}_{\beta}^{\Hh}(\mg)}
\nc{\KbHm}{\mathsf{K}_{\beta}^{\Hh}(\mg)}

\nc{\kgss}{{\kg_{\mathsf{ss}}}}
\nc{\Kss}{{\K_{\mathsf{ss}} } }
\nc{\Z}{{\mathsf{Z}}}

\nc{\mug}{{\mu^\ggo}}
\nc{\kgub}{{\kg_{\ug_{\Beta}}}}
\nc{\qgb}{{\qg_\Beta}}
\nc{\qgHb}{{\qg_\Beta^\Hh}}

\nc{\mcal}{\mca_\mathsf{left}}
\nc{\detb}{\det_{\Beta_\mg}}
\nc{\ggl}{\mathfrak{g}}
\nc{\ppm}{\mathfrak{p}}
\nc{\GG}{G}
\nc{\Gm}{\Gl(\mg)}
\nc{\Gg}{\Gl(\ggl)}
\nc{\glgg}{{\mathfrak{gl}(\mathfrak{g})}}
\nc{\GH}{\Gl^{\! \Hh}}
\nc{\GHm}{\Gl^{\! \Hh} (\mg)}
\nc{\GHmk}{\Gl^H(\mg_k)}
\nc{\gHm}{{\mathfrak{gl}^H(\mathfrak{m})}}
\nc{\OHm}{\Or^{\Hh}(\mg)}
\nc{\sogHm}{ {\mathfrak{so}^H(\mathfrak{m})} }
\nc{\Vm}{V(\mg)}
\nc{\Vg}{V(\ggl)}
\nc{\Vmi}{V(\mg_\infty)}
\nc{\Om}{\Or(\mg)}
\nc{\Og}{\Or(\ggl)}
\nc{\Symm}{{\rm Sym}(\mg)}
\nc{\Symg}{{\rm Sym}(\ggl)}
\nc{\gm}{\mathfrak{gl}(\mg)}
\nc{\som}{\sog(\mg)}
\nc{\sogg}{{\mathfrak{so}(\mathfrak{g})}}
\nc{\hml}{{\mu^{\ggl}}}
\nc{\ml}{{\mu^{\ggl}_{\mg}}}
\nc{\Ol}{{\mathcal{O}_{\ggl}}}
\nc{\OlH}{{\mathcal{O}^H_{\ggl}}}
\nc{\Vn}{V(n)}
\nc{\SymV}{{\rm Sym}(V)}
\nc{\mub}{{\bar \mu}}
\nc{\mumg}{{\mu_\mg}}
\nc{\Betam}{{\Beta_\mg}}
\nc{\Betag}{{\Beta_\ggo}}
\nc{\Betagp}{{\Beta_\ggo^+}}
\nc{\ii}{{\mathrm{i}}}
\nc{\Nrm}{{\mathrm{N}}}
\nc{\Srm}{{\mathrm{S}}}
\nc{\iR}{{\ii\RR}}

\nc{\Vs}{{V(\sg)}}
\nc{\Syms}{{\Sym(\sg)}}
\nc{\glgs}{{\glg(\sg)}}

\nc{\slgb}{{\slg_\Beta}}

\nc{\Vzerog}{V_{\Beta_\ggo^+}^{0}}
\nc{\Vnng}{V_{\Beta_\ggo^+}^{\geq 0}}
\nc{\Vnnssg}{U_{\Beta_\ggo^+}^{\geq 0}}
\nc{\Vzerossg}{U_{\Beta_\ggo^+}^{0}}

\nc{\musol}{{\mu_{\mathsf{sol}}}}

\nc{\Go}{\mathsf{\bar G}} 
\nc{\No}{\mathsf{\bar N}} 
\nc{\Hho}{ \mathsf{\bar H}}
\nc{\ggoo}{\mathfrak{\bar g}}
\nc{\mgo}{{\mathfrak{\bar m}}}
\nc{\hgo}{ \mathfrak{\bar h}}
\nc{\mcvo}{{\operatorname{H}}} 
\nc{\mcvot}{{\operatorname{H}_t}}
\nc{\mcvoz}{\operatorname{ H}_0} 
\nc{\mgos}{\mathfrak{\bar m^*}}
\nc{\ngoo}{\mathfrak{\bar n}}
\nc{\QbHmo}{\mathsf{Q}_{\beta}^{\Hho}(\mgo)}
\nc{\KbHmo}{\mathsf{K}_{\beta}^{\Hh}(\mgo)}
\nc{\Omo}{\Or(\mgo)}
\nc{\Gmo}{\Gl(\mgo)}

\nc{\Gl}{\mathsf{GL}} 	\nc{\Or}{\mathsf{O}}  	\nc{\SO}{\mathsf{SO}}	\nc{\Sl}{\mathsf{SL}}
\nc{\G}{\mathsf{G}} 
	\nc{\K}{\mathsf{K}}  	\nc{\T}{\mathsf{T}} 	\nc{\Ll}{\mathsf{L}}
\nc{\Hh}{\mathsf{H}}
\nc{\Zz}{\mathsf{Z}}

 	\nc{\PPP}{\mathsf{P}}  	\nc{\Ss}{\mathsf{S}} 	\nc{\U}{\mathsf{U}}
\nc{\Aa}{\mathsf{A}} 	\nc{\N}{\mathsf{N}}		\nc{\B}{\mathsf{B}}		\nc{\Rr}{\mathsf{R}}
\nc{\Q}{\mathsf{Q}}
\nc{\Uu}{\mathsf{U}}  \nc{\Vv}{\mathsf{V}}

\nc{\Qb}{\mathsf{Q}_\Beta} 				\nc{\Hb}{\mathsf{H}_\Beta} 			\nc{\Ub}{\mathsf{U}_\Beta} 
\nc{\Gb}{\mathsf{G}_\Beta} 				\nc{\Kb}{\mathsf{K}_\Beta} 			\nc{\Slb}{\Sl_{\Beta}} 
\nc{\Slbm}{\Sl_{\Beta_\mg}^H} 			\nc{\slbm}{\slg_{\Beta_\mg}^H(n)} 	\nc{\Beta}{{\beta}}
\nc{\Alpha}{A}							\nc{\Vr}{V_{\Beta^+}^{r}}   		\nc{\Vzero}{V_{\Beta^+}^{0}}
\nc{\Vnn}{V_{\Beta^+}^{\geq 0}}			\nc{\Vnnss}{U_{\Beta^+}^{\geq 0}}	\nc{\Vzeross}{U_{\Beta^+}^{0}}
\nc{\Vnnt}{V_{\tilde\Beta^+}^{\geq 0}}	\nc{\Betap}{ {\Beta + \Vert{\Beta}\Vert^2 \Id} }
\nc{\Ap}{ {A + \Vert{A}\Vert^2 \Id} }

\nc{\Gs}{{\Gl(\sg)}}  \nc{\Os}{{\Or(\sg)}}
\nc{\gsol}{{g_{\mathsf{sol}}}} \nc{\bgsol}{{\bar g_{\mathsf{sol}}}}

\nc{\GGs}{S}
\nc{\ggs}{\mathfrak{s}}

\nc{\ggo}{\mathfrak{g}}

\nc{\ggob}{\overline{\mathfrak{g}}}
\nc{\lamg}{\Lambda^2\ggo^*\otimes\ggo}
\nc{\gkp}{(\ggo=\kg\oplus\pg,\ip)} \nc{\ukh}{(\ug=\kg\oplus\hg,\ip)}
\nc{\tgkp}{(\tilde{\ggo}=\kg\oplus\pg,\ip)}

\nc{\fg}{\mathfrak{f}}  	\nc{\vg}{\mathfrak{v}} 		\nc{\wg}{\mathfrak{w}} 		\nc{\zg}{\mathfrak{z}} 
\nc{\ngo}{\mathfrak{n}} 	\nc{\kg}{\mathfrak{k}} 		
\nc{\mg}{\mathfrak{m}}

\nc{\bg}{\mathfrak{b}}
\nc{\sog}{\mathfrak{so}} 	\nc{\sug}{\mathfrak{su}} 	\nc{\spg}{\mathfrak{sp}} 	\nc{\slg}{\mathfrak{sl}}
\nc{\glg}{\mathfrak{gl}} 	\nc{\cg}{\mathfrak{c}} 		\nc{\rg}{\mathfrak{r}}  	
\nc{\hg}{\mathfrak{h}} 
 
\nc{\tgo}{\mathfrak{t}} 	\nc{\ug}{\mathfrak{u}} 		\nc{\dg}{\mathfrak{d}} 		\nc{\ag}{\mathfrak{a}} 
\nc{\pg}{\mathfrak{p}} 		\nc{\sg}{\mathfrak{s}} 		\nc{\affg}{\mathfrak{aff}} 	\nc{\qg}{\mathfrak{q}} 
\nc{\eg}{\mathfrak{e}}		\nc{\Xg}{\mathfrak{X}} 		\nc{\lgo}{\mathfrak{l}} 	\nc{\tg}{\mathfrak{t}}

\nc{\pca}{\mathcal{P}} 		\nc{\nca}{\mathcal{N}} 		\nc{\lca}{\mathcal{L}} 		\nc{\oca}{\mathcal{O}} 
\nc{\mca}{\mathcal{M}} 		\nc{\tca}{\mathcal{T}} 		\nc{\aca}{\mathcal{A}} 		\nc{\cca}{\mathcal{C}} 
\nc{\gca}{\mathcal{G}} 		\nc{\sca}{\mathcal{S}} 		\nc{\hca}{\mathcal{H}} 		\nc{\bca}{\mathcal{B}} 
\nc{\dca}{\mathcal{D}} 		\nc{\fca}{\mathcal{F}} 		\nc{\Qca}{\mathcal{Q}}

\nc{\im}{\mathtt{i}}    \renewcommand{\Im}{{\rm Im}}

\nc{\RR}{{\mathbb R}} \nc{\HH}{{\mathbb H}} \nc{\CC}{{\mathbb C}} \nc{\ZZ}{{\mathbb Z}}
\nc{\FF}{{\mathbb F}} \nc{\NN}{{\mathbb N}} \nc{\QQ}{{\mathbb Q}} \nc{\PP}{{\mathbb P}}
\nc{\KK}{{\mathbb K}}

\nc{\vs}{\vspace{.2cm}} \nc{\vsp}{\vspace{1cm}} 
\nc{\ip}{{\langle \,\cdot \,,\cdot \,\rangle }}
 \nc{\la}{\langle} \nc{\ra}{\rangle} \nc{\unm}{\tfrac{1}{2}}
\nc{\unc}{\tfrac{1}{4}} \nc{\und}{\tfrac{1}{16}} \nc{\no}{\vs\noindent}
\nc{\lam}{\Lambda^2(\RR^n)^*\otimes\RR^n} \nc{\tangz}{{\rm T}^{\rm Zar}}

\nc{\nor}{{\sf n}}  \nc{\mum}{/\!\!/} \nc{\kir}{/\!\!/\!\!/}
\nc{\Ri}{\tfrac{4\Ric_{\mu}}{||\mu||^2}} \nc{\ds}{\displaystyle}
\nc{\ben}{\begin{enumerate}} \nc{\een}{\end{enumerate}} \nc{\f}{\frac}
\nc{\lb}{{[\cdot,\cdot]}} \nc{\isn}{\tfrac{1}{||v||^2}}

\nc{\wt}{\widetilde}
\nc{\raw}{\rightarrow} \nc{\lraw}{\longrightarrow} \nc{\hqn}{\mathcal{H}_{q,n}}
\nc{\minimatrix}[4]{\left[\begin{smallmatrix} {#1} & {#2} \\ {#3} & {#4} \end{smallmatrix}\right]}
\nc{\twomatrix}[4]{\left[\begin{array}{cc} {#1} & {#2} \\ {#3} & {#4} \end{array} \right]}
\nc{\threematrix}[9]{\left[\begin{array}{ccc} {#1} & {#2} & {#3} \\ {#4} & {#5} & {#6}\\ {#7} & {#8} & {#9} \end{array} \right]}
\nc{\mut}{\tilde{\mu}} \nc{\mur}{{\mu_r}} \nc{\mutr}{{\tilde{\mu}_r}}

\nc{\alert}{\color{blue}}

\nc{\ad}{\operatorname{ad}}  	\nc{\Aut}{\operatorname{Aut}}   	\nc{\Inn}{\operatorname{Inn}}   
\nc{\Lie}{\operatorname{Lie}} 	\nc{\Ad}{\operatorname{Ad}} 		\nc{\Der}{\operatorname{Der}} 
\nc{\rad}{\operatorname{rad}} 	\nc{\kf}{\operatorname{B}} 

\nc{\End}{\operatorname{End}} \nc{\rank}{\operatorname{rank}} \nc{\Ker}{\operatorname{Ker}} \nc{\tr}{\operatorname{tr}} \nc{\Sym}{\operatorname{Sym}} \nc{\diag}{\operatorname{diag}} \nc{\proy}{\operatorname{pr}} \nc{\Adj}{\operatorname{Adj}} \nc{\proj}{\operatorname{pr}} \nc{\Id}{{\operatorname{Id}}} \nc{\Span}{\operatorname{span}}  \nc{\spec}{{\operatorname{sp}}}

\nc{\Hess}{\operatorname{Hess}}  		\nc{\dif}{\operatorname{d}} 	\nc{\sen}{\operatorname{sen}} 
\nc{\grad}{\operatorname{grad}} 		\nc{\Order}{\operatorname{O}} 	\nc{\divg}{\operatorname{div}}
\nc{\dd}{{\rm d}}  						\nc{\ddt}{\tfrac{{\rm d}}{{\rm d}t}}        
\nc{\dds}{\tfrac{{\rm d}}{{\rm d}s}}	\nc{\ddtbig}{\frac{{\rm d}}{{\rm d}t}}      
\nc{\dpar}{\tfrac{\partial}{\partial t}}    

\nc{\Iso}{\operatorname{Iso}} 	\nc{\Diff}{\operatorname{Diff}} \nc{\Rc}{\operatorname{Rc}} 
\nc{\Ricci}{\operatorname{Ric}}
\nc{\ric}{\operatorname{ric}}  \nc{\ricci}{\operatorname{ric}} 
\nc{\riccim}{\operatorname{ric}^\star} 
\nc{\Riem}{\operatorname{Rm}} \nc{\scal}{\operatorname{scal}} \nc{\scalm}{\operatorname{scal}^\star} \nc{\Riccim}{\operatorname{Ric}^{\star}} \nc{\tang}{\operatorname{T}} \nc{\vol}{\operatorname{vol}} 
\nc{\mcv}{{\operatorname{H}}} 

\nc{\inj}{\operatorname{inj}}
\nc{\isog}{\mathfrak{iso}}

\nc{\mm}{\operatorname{M}} \nc{\CH}{\operatorname{CH}} \nc{\Irr}{\operatorname{Irr}} \nc{\mcc}{\operatorname{mcc}} \nc{\Sb}{\mathcal{S}_\Beta} \nc{\mmm}{\operatorname{m}}   

\nc{\zero}{ {\backslash \{0\} } }
\nc{\normmm}{{\rm F}}
\nc{\ipp}{\la\,\cdot \,,\cdot\,\ra^*_g}
\nc{\ippk}{\la\,\cdot \,,\cdot\,\ra^*_{g_k}}
\nc{\ippi}{\la\,\cdot \,,\cdot\,\ra^*_{g_\infty}}
\nc{\ipnew}{\la \la \cdot , \cdot \ra\ra}

\nc{\II}{{\mathbb I}}

\theoremstyle{plain}
\newtheorem{theorem}{Theorem}[section]
\newtheorem{proposition}[theorem]{Proposition}
\newtheorem{corollary}[theorem]{Corollary}
\newtheorem{lemma}[theorem]{Lemma}
\newtheorem{teointro}{Theorem}
\newtheorem{corintro}[teointro]{Corollary}

\theoremstyle{definition}
\newtheorem{definition}[theorem]{Definition}

\theoremstyle{remark}
\newtheorem{remark}[theorem]{Remark}

\begin{document}
\begin{titlepage}

\title{Homogeneous Einstein metrics on Euclidean spaces are Einstein solvmanifolds}

\author{Christoph B\"ohm}	
\address{University of M\"unster, Einsteinstra{\ss}e 62, 48149 M\"unster, Germany}
\email{cboehm@math.uni-muenster.de}

\author{Ramiro A.~ Lafuente} 
\address{School of Mathematics and Physics, The University of Queensland, St Lucia QLD 4072, Australia}
\email{r.lafuente@uq.edu.au}

\begin{abstract}
We show that 
 homogeneous Einstein metrics on Euclidean spaces
 are  Einstein solv\-manifolds,
using that they admit  
periodic,  integrally minimal foliations
by homogeneous hypersurfaces.  For the geometric flow induced  by the orbit-Einstein condition,  we construct a Lyapunov function
based on curvature estimates which come from real GIT.
\end{abstract}

\end{titlepage}


\maketitle
\setcounter{page}{1}
\setcounter{tocdepth}{0}

A Riemannian manifold $(M^n,g)$ is called Einstein, if its Ricci tensor
satisfies $\ric_g=\lambda\cdot g$, for some Einstein constant $\lambda \in \RR$.  As is well known, for $n\leq 3$ this implies constant sectional curvature. Topological obstructions to the existence
of compact Einstein $4$-manifolds are known, see \cite{Tho1969}, \cite{Hit1974}, \cite{Leb1}, \cite{Leb2} and \cite{And2010}.   In dimensions $n\geq 5$,
compact simply-connected Einstein manifolds
with positive Einstein constant are also topologically obstructed, see \cite{Sto},
whereas for  negative Einstein constant no such obstruction is known. 

Many examples of compact Einstein manifolds have been constructed using bundle, symmetry and holonomy assumptions:
see  \cite{Bss}, \cite{Wang0}, \cite{Wang},  \cite{Joybook}
 and references therein.
 Among many others, we mention homogeneous Einstein metrics \cite{WZ86}, \cite{BWZ}, 
Einstein metrics on  spheres \cite{Bohm98}, \cite{BGK}, \cite{FH2017}, Ricci-flat manifolds with holonomy
$\G_2$ and ${\rm Spin}(7)$ \cite{Joy1}, \cite{Joy2},
 and   K\"ahler-Einstein manifolds, whose classification could be completed recently
\cite{Aub}, \cite{Yau},
\cite{CDS1},\cite{CDS2}, \cite{CDS3}, \cite{Tian2015}.

Non-compact Einstein manifolds have non-positive Einstein constant
and, in contrast to the compact case,
no topological obstruction is known in dimensions $n\geq 4$. 
As in the compact case, there exist plenty of examples, see
\cite{Bry}, \cite{BrSa}, \cite{AKL}, \cite{Biq2013}, \cite{Biq2017} and  the survey articles mentioned above. These include irreducible non-compact symmetric spaces, and more generally non-compact homogeneous Einstein spaces. The latter are flat for zero Einstein constant 
by \cite{AlkKml}, and for 
negative Einstein constant, Alekseevskii's conjecture states that $M^n$ is  
 diffeomorphic to  
 Eucli\-dean space $\RR^n$:  see \cite{Alek75}, \cite{Bss}.


The main result of this paper is 

\begin{teointro}\label{thm_main}
Homogeneous Einstein metrics on  $\RR^n$  are isometric to
Einstein solvmanifolds.
\end{teointro}

Recall that a simply-connected solvmanifold is
a solvable Lie group endowed with a left-invariant metric.
Theorem \ref{thm_main}
 was known for Ricci-flat homogeneous spaces \cite{AlkKml},
 for  homogeneous $\RR$-bundles over irreducible  Hermitian symmetric spaces, see \cite{BB} and
Remark \ref{rem_euclhom},
and in dimensions $n\leq 5$ and $n=7$ \cite{AL16}, leaving open
the case of the $6$-dimensional universal cover $\Sl_*(2,\RR)^2$ 
of $\Sl(2,\RR)^2$. In this direction, we would like to mention 
 that our methods also yield new results beyond Theorem \ref{thm_main}, such as the
 non-existence of left-invariant Einstein metrics on $\Sl_*(2,\RR)^k$ for all $k\geq 2$: see Corollary \ref{cor_nonexist}.

The main difficulty in proving Theorem \ref{thm_main} is that
 Euclidean spaces admit
 extremely different presentations as a homogeneous space $\G/\Hh$. Firstly,
$\RR^n$ is a solvmanifold, which  in dimension $n=3$ already provides uncountably many algebraically distinct choices \cite{Bian1898}.
Diametrically opposed, it is also a homogeneous space for $\G$ semisimple, 
such as $\RR^{3k}=\Sl_*(2,\RR)^k$, $k\geq 1$,
or $\RR^{2m+1}$ being the universal cover of
$\SO(m,2)/\SO(m)$, $m\geq 3$: see Remark \ref{rem_euclhom}
for further examples. Finally and more generally,  semidirect products  of the above cases occur.

  
Next, we restate the following important classification result of Lauret:

\begin{teointro}[\cite{standard}]\label{thm_Lauret}
Einstein solvmanifolds are standard.
\end{teointro}

The standard condition, described below, is an algebro-geometric condition introduced and extensively studied by Heber in \cite{Heb}. 
For standard Einstein solvmanifolds of fixed dimension, he showed
 finiteness of the eigenvalue
type of the \emph{modified Ricci curvature}, see below,
a result intimately related to
the finiteness of critical values of a (real) moment map.
Even though a complete classification  is out of reach at the moment, see \cite{cruzchica}, Theorem \ref{thm_main} together with \cite{standard} and \cite{Heb} would provide a very precise understanding of \emph{all} non-compact homogeneous Einstein manifolds, provided the Alexseevskii conjecture holds true.

 
A Ricci soliton $(M^n,g)$ is a Riemannian manifold satisfying
$\ric_g+\mathcal{L}_X g= \lambda\cdot g$, $\lambda \in \RR$, where
 $\mathcal{L}_X g$ denotes the Lie derivative of $g$ in the direction of a smooth vector field $X$ on $M^n$. 
Recall now that 
a homogeneous Ricci soliton on $\RR^n$
gives rise to a homogeneous Einstein space $(\RR^{n+1},\hat g)$:
see \cite{alek}, \cite{HePtrWyl}. Since 
by work of Jablonski \cite{Jab15} Einstein solvmanifolds
are \emph{strongly solvable}  spaces,  from
Theorem \ref{thm_main} we deduce

\begin{corintro}\label{cor_main1}
Homogeneous Ricci solitons on $\RR^n$ are isometric to solvsolitons.
\end{corintro}

Solvsolitons,  introduced in \cite{solvsolitons}, are homogeneous Ricci solitons admitting a transitive solvable Lie group $\Ss$ of isometries. It is shown in \cite{Jbl2015} that  $\Ss$ may be chosen to be simply-transitive,  such that the Ricci endomorphism 
satisfies $\Ricci_g - \lambda \cdot \Id ={\rm D}$, where ${\rm D}$
is a derivation of the Lie algebra $T_e\Ss$. It follows that the 
corresponding Ricci flow solution is driven by the one-parameter group of automorphisms of $\Ss$ corresponding to ${\rm D}$.
Automorphisms  are precisely those diffeomorphisms  which preserve the space of left-invariant metrics on $\Ss$.

A simply-connected homogeneous space is diffeomorphic to $\RR^n$
if and only if \emph{all} its homogeneous metrics have
 non-positive scalar curvature  \cite{BB}. As a consequence 
homogeneous Ricci flow solutions on $\RR^n$ always exist for all positive 
times  \cite{scalar}. 
By \cite{BL17}  and Corollary \ref{cor_main1}  it then follows
that any parabolic blow-down of such a solution subconverges to a limit solvsoliton in pointed Cheeger-Gromov topology. 
Notice that for the semisimple Lie group  
$\Sl_*(2,\RR)^k$ 
this means that for generic left-invariant
initial metrics the Lie group structure 
of their full isometry group changes completely when passing to a limit.

Turning to the proof of Theorem \ref{thm_main}, we would like to mention that all the algebraic structure results for non-compact homogeneous Einstein manifolds developed in \cite{alek}, \cite{JblPet14}
and \cite{AL16}  yield for instance no information whatsoever for the presentation  $\RR^{3k}=\Sl_*(2,\RR)^k$. Therefore,
 a purely algebraic proof of Theorem \ref{thm_main} is elusive at the moment.
  To overcome this, we prove that homogeneous Euclidean Einstein spaces
  admit  cohomogeneity-one actions by 
non-unimodular subgroups $\Go$ of $\G$, with orbit space $\RR$.
We then show \emph{periodicity} of  the corresponding foliation by orbits, 
  meaning that after passing to a quotient, which is still homogeneous, the orbit space becomes $S^1$, and 
  that the integral of the
mean curvature of the orbits over the orbit space vanishes, a condition
we call  \emph{integral minimality}.
 This then reduces the proof of Theorem \ref{thm_main} essentially
  to the following second main result
 of this paper.

\begin{teointro}\label{thm_main2}
Suppose that $(M^n, g)$ admits an effective, cohomogeneity-one action of  a Lie group $\Go$ with closed, integrally minimal orbits and $M^n/\Go = S^1$. If in addition $(M^n,g)$ is \emph{orbit-Einstein} with
negative Einstein constant, then all orbits are standard homogeneous spaces.
\end{teointro}

A cohomogeneity-one manifold $(M^n, g)$ is called
\emph{orbit-Einstein} with negative Einstein constant, 
if for $\lambda <0$ we have 
$\ricci_{g}(X,X) = \lambda\cdot g(X,X)$
 for all vectors $X$ tangent to orbits.
Generalizing \cite{Heb}, we say that a homogeneous space $(\Go/\Hho,\bar g)$  is \emph{standard},
if the Riemannian submersion induced by the free isometric action of the maximal connected normal nilpotent subgroup $\No \leq \Go$  on $\Go/\Hho$ has \emph{integrable} horizontal distribution: 
see Definition \ref{def_stand}. 

Finally, let us mention that Theorem \ref{thm_main2}
  immediately implies \mbox{Theorem \ref{thm_Lauret}} when
applying it  to the Riemannian product of a circle and an Einstein solvmani\-fold. 
Other non-existence results on homogeneous and cohomogeneity-one Einstein metrics include
\cite{WZ86},  \cite{Bhm99}, \cite{Nkn2},  \cite{Bhm05}, \cite{JblPet14}, \cite{AL16}.

We turn now to the proof of Theorem \ref{thm_main2}.
Since the $\Go$-orbits form a family of
 equidistant hypersurfaces in $M^n$, we write
$g =dt^2 +g_t$, $t \in \RR$,
for a smooth curve of homogeneous metrics $g_t$
on $\bar M=\Go/\Hho$.  By the
 Gau\ss{} equation, the Riccati equation  and
 \cite{Bhm1999},  the orbit-Einstein equation \eqref{eqn_Riccati} on $(M^n,g)$ with Einstein constant $-1$
can be considered as the following `second order Ricci flow' on the space of 
homogeneous metrics on $\bar M$:
\begin{eqnarray}\label{eqn_2RF}
 \tfrac{D^2 }{dt^2} g_t =-(\tr L_t) \cdot g_t' +2  \ric_{g_t} +2 g_t\,.
\end{eqnarray}
Here $\tr L_t$ denotes the mean curvature of an orbit
and $\tfrac{D}{dt}$ the covariant derivative of the symmetric
 metric on the symmetric space of $\Go$-homo\-ge\-neous metrics on $\Go/\Hho$.
 
We decompose
the Ricci tensor $\ricci_{\bar g}$ of a homogeneous space
$(\bar M=\Go/\Hho,\bar g)$  as a sum of two tensors, one tangent to the ${\rm Diff}_{\Go}(\bar M)$-orbit 
through $\bar g$, and  another one, the \emph{modified Ricci curvature}
 $\riccim_{\bar g}$, orthogonal to it. 
  Here, 
  ${\rm Diff}_{\Go}(\bar M)$ denotes the set of 
  those diffeomorphisms 
  of $\bar M$ which preserve the space of  $\Go$-invariant metrics when acting by pull-back;
algebraically these are just the automorphisms of $\Go$
preserving $\Hho$.
Denoting $\scalm_{\bar g} = \tr_{\bar g}\riccim_{\bar g}$,
we set 
\[
  h(\bar g):=\tfrac{1}{2}\cdot (\scalm_{\bar g}-\scal_{\bar g})\geq  0\,,
\] 
see Lemma \ref{lem_h} and Remark \ref{rem_defh}. 
 Algebraically,
$h=0$ if and only if the group $\Go$ is unimodular.

Using the orbit-Einstein equation for $(M^n,g)$ and the compactness
of the orbit space
we establish a maximum principle for the real-valued function 
$h(t):=h(g_t)$:
see Lemma \ref{lem_h} and Lemma \ref{lem_betapgeqh}.
This yields an upper bound for  $2h$ given by 
$\tr \Beta^+=n-1/\Vert \Beta\Vert^2\geq 0$. Here
 $\Beta$ is the stratum label of the homogeneous space
$\Go/\Hho$, a self-adjoint endomorphism, coming from the
Morse-type,  Ness-Kirvan stratification of the space of Lie 
brackets on $T_e \Go$ introduced by Lauret \cite{standard}: 
see  Appendix \ref{app_beta}  and \cite{BL17}.
From this  a priori estimate for $h$ 
we establish in \mbox{Lemma \ref{lem_lmonotone}}
the existence of a Lyapunov function
for the orbit-Einstein equation. Using that the orbits are integrally minimal
it follows that this Lyapunov function is periodic,  hence constant. As a consequence, several inequalities
become equalities and Theorem \ref{thm_main2} follows.

Finally, we indicate how Theorem \ref{thm_main2} implies
Theorem \ref{thm_main}. Using the standard condition for the given homogeneous Einstein space $(\RR^n=\G/\Hh,g)$, see \cite{alek},
and all the codimension-one orbits, one shows that the simple factors of a Levi factor $\Ll \leq \G$
are pairwise orthogonal: see Proposition \ref{prop_simpleperp}.
Together with the condition $\G/\Hh \simeq \RR^n$, this implies that the induced metric on $\Ll/\Hh$ is \emph{awesome}, that is, it admits an orthogonal Cartan decomposition.
This leads to a contradiction by \cite{Nkn2}, unless
the Levi factor is trivial, in which case $\G$ is solvable.

The article is organized as follows. In Section 1 we discuss the cohomogeneity one Einstein equation and derive the evolution equation
for $h(t)$  induced by \eqref{eqn_2RF}. In Section  \ref{sec_maxp} we prove Theorem \ref{thm_main2} by
establishing a maximum principle for $h$ and by defining 
a new Lyapunov function for \eqref{eqn_2RF}.
In Section  \ref{sec_homspaces} we introduce the assumption $M\simeq \RR^n$ which will be used from here on, and characterise  certain minimal presentations of homogeneous Euclidean spaces
and prove their periodicity in Section \ref{sec_period}.
In Section \ref{sec_integralm} integral minimality of the so called 
Levi presentations is shown, and finally,
in Section \ref{sec_proofmain} we prove Theorem \ref{thm_main}  and Corollary \ref{cor_main1}.
In Appendix \ref{app_homogeneous}
we address elementary algebraic properties of homogeneous spaces, and Appendix \ref{app_beta} contains the necessary preliminaries on GIT and the $\beta$-endomorphism associated to a homogeneous space.

\vs \noindent {\it Acknowledgements.} The authors would like to thank Michael Jablonski for his useful comments on a first version of this article.

\section{The cohomogeneity-one Einstein equation}\label{sec_cohom1}

Let $\Go$ be a connected Lie group acting 
on a complete, connected Riemannian manifold $(M^n,g)$ 
by $(g,p)\mapsto g \cdot p$ for $g\in \Go$ and  $p\in M^n$.
We assume that the action is
proper, isometric, effective and with cohomogeneity one.
By \cite{Kos65},
properness ensures that the orbits are embedded submanifolds, and that the isotropy subgroups are compact. 

Let us in addition assume  that all orbits are principal. 
Then by \cite{Pal61},
the orbit space ${M^n}/\Go$ is diffeomorphic to  $\RR$ or $S^1$. 
Choose $\gamma:\RR \to M^n$ a unit speed geodesic intersecting all orbits orthogonally, and set $\Sigma_t := \Go \cdot \gamma(t)$.
Let $N$ denote the unit normal vector field of the foliation
$\Sigma_t$ of $M^n$ with  
$\gamma'(t)=N_{\gamma(t)}$  for all $t \in \RR$. Then, all integral 
curves of $N$ are unit speed geodesics intersecting all $\Go$-orbits
$\Sigma_t$ orthogonally.

Infinitesimally, the $\Go$-action induces a Lie algebra homomorphism 
$\ggoo := \Lie(\Go) \to \Xg(M^n)$  assigning to each $X\in \ggoo$ a Killing field on $(M^n, g)$, also denoted by $X$, and given by
\begin{equation}\label{eqn_actionfield}
		X_p := \ddt\big|_{t=0} \exp(tX) \cdot  p, \qquad p\in M^n\,. 
\end{equation}
 Since the flow of Killing fields consists of isometries, it maps geodesics  onto geodesics, and consequently
$[X,N]=0$ where $[ \, \cdot \, , \cdot  \,]$ denotes the Lie bracket of smooth vector fields on $M^n$.

Since  $\Sigma_0$ is a principal orbit, the
isotropy subgroup $\Hho:=\Go_{\gamma(0)}$ at $\gamma(0)$ fixes all points $\gamma(t)$, hence $\Go_{\gamma(t)}=\Hho$  for all $t$. Set $\hgo :=\Lie(\Hho)$ and let $\ggoo = \hgo \oplus \mgo$ be the `canonical' reductive decomposition: $\mgo$ is the orthogonal complement of  $\hgo$ in $\ggoo$ with respect to the Killing form 
$\kf_\ggoo$ of $\ggoo$.
By \eqref{eqn_actionfield} we obtain for all $t$ an identification 
\begin{equation}\label{eqn_identifggo}
	T_{\gamma(t)} \Sigma_t \simeq \mgo\,,
\end{equation}
which will be used constantly in what follows. In particular, the metric $ g$
on $M^n$ induces a family of $\Ad(\Hho)$-invariant scalar products $( g_t)_{t\in \RR}$ on $\mgo$. 

Let $L_t \in \End(\mgo)$ denote the shape operator
of $\Sigma_t$ at $\gamma(t)$ evaluated on Killing fields:
\begin{eqnarray*}
 g_t(L_t X, Y) =  g(\nabla_X N, Y)_{\gamma(t)} = - g({\gamma'(t)}, \nabla_X Y ),  \qquad X,Y\in  \mgo,
\end{eqnarray*}
since $ g(N,Y)\equiv 0$. Here
 $\nabla$ denotes the Levi-Civita connection of $(M^n, g)$.

\begin{lemma}\label{lem_ddtgt}
The scalar products $(g_t)_{t\in \RR}$ on $\mgo$ evolve by 
$g_t' ( \,\cdot \, ,\, \cdot )
	=2\, g_t(L_t \,\cdot \, ,\, \cdot )$.
\end{lemma}
\begin{proof}
Since $g([X,Y],\gamma'(t))_{\gamma(t)}=0$
for $X, Y \in \mgo$,
$L_t$ is $g_t$-symmetric. Thus,
\begin{align*}
	\ddt \,  g(X, Y)_{\gamma(t)} = 
	  g(\nabla_{\gamma'(t)} X , Y) +  g(X, \nabla_{\gamma'(t)} Y)  
	 = -2g_t (L_t X, Y)\,,
\end{align*}
where we have used that $\nabla_{\cdot} X, \nabla_{\cdot} Y$ are 
skew-symmetric.
\end{proof}

Let $\ric_t \in \Sym^2(\mgo)$ denote the Ricci curvature of $(\Sigma_t,g_t)$ at $\gamma(t)$, and $\Ricci_t \in \End(\mgo)$ the $g_t$-symmetric Ricci operator defined by
\[
	\ric_t(X,Y) = g_t(\Ricci_t X, Y), \qquad X, Y\in \mgo\,.
\]

\begin{proposition}\label{prop_ddtLt}
If $\ric_{g}(X,X) = -{g}(X,X)$ for all $X \in \mgo$, then
\begin{align}
	L_t' &=-(\tr L_t) \cdot L_t + \Ricci_t +  \Id_\mgo\,, \label{eqn_Riccati} 
\end{align}
\end{proposition}

\begin{proof}
The Gauss equation for the hypersurface $\Sigma_t \subset M^n$ gives
\[
	\ric_{ g} (X,Y) = \ric_t(X,Y) +  g( R_N X, Y)   + g(L_t X, L_t Y) 
	- (\tr L_t) \cdot   g(L_t X, Y),
\]
for Killing fields $X,Y\in \mgo$, where $ R_N X = \Riem_{ g}(X, N) N$. The claim now follows from the Riccati equation 
$\nabla_N L + L +  R_N = 0$ using that $L'=\nabla_N L$; see 
 \cite{EscWng00}.
\end{proof}

A Lie group $\Go$ is unimodular if
$\tr \ad(X) = 0$ for all $X\in \ggoo$,
 $\ad (X):\ggoo \to \ggoo,\,\,Y \mapsto [X,Y]$. 
If $\Go$ is not
unimodular, then it contains a  normal, unimodular subgroup
$\Go_u$  with Lie algebra
$\ggoo_u := \{X\in \ggoo : \tr \ad (X) = 0 \}$. It follows that
 each $\Go$-orbit
 $(\Sigma_t,g_t)$ is foliated by pairwise isometric $\Go_u$-orbits
 of codimension one with a $\Go$-invariant  mean curvature vector field.
At the point $\gamma(t)\in \Go_u\cdot \gamma(t) \subset\Sigma_t$ 
we can consider 
this mean curvature vector as 
 $\mcvot \in \mgo$, using \eqref{eqn_identifggo}. An algebraic description of
$\mcvot$ is given in \cite[Lemma 7.32]{Bss}:
\begin{equation}\label{eqn_mcv}
	g_t(\mcvot, X) = \tr \ad(X), \qquad \forall X\in \mgo.
\end{equation}

\begin{lemma}\label{lem_h}
The function 
 \[
 	h: \RR \to \RR\,\,;\,\,\,t \mapsto  \unm \,  g_t(\mcvot, \mcvot)\,
 \]
satisfies the following evolution equations:
\begin{align*}
	 \mcvo_t' = -2 \, L_t  \mcvot \, ,\quad
	 h' =  - g_t(L_t \mcvot, \mcvot)\,,\quad  
	 h'' = 2\, \Vert L_t \mcvot\Vert_t^2 -  \,g_t (L_t' \mcvot, \mcvot).
\end{align*}
Moreover, if $\ricci_g (X,X) = -g(X,X)$ for all $X\in \mgo$  then 
\[
	h'' = 2 \, \Vert L_t \mcvot\Vert_t^2 - (\tr L_t) \, h' -  2 \, h - \ricci_t (\mcvot, \mcvot).
\] 
\end{lemma}
\begin{proof}
Differentiating \eqref{eqn_mcv} and using Lemma \ref{lem_ddtgt} we get
$0 = g_t(\mcvo_t', X) + 2 \, g_t(L_t \mcvot, X)$
for all $X \in \mgo$, which gives the first formula. Using that, we may compute
\[
	h' = \unm \, g_t'(\mcvot, \mcvot) +  g_t (\mcvo_t', \mcvot) 
	= -  g_t( L_t \mcvot, \mcvot).
\]
Differentiating once more and using Lemma \ref{lem_ddtgt} again we obtain
\begin{align*}
	h'' = - g_t'(L_t \mcvot, \mcvot) - g_t(L_t' \mcvot, \mcvot) - 2 \, g_t (L_t \mcvo_t', \mcvot) 
	=  2 \, \Vert L_t \mcvot\Vert_t^2 - g_t (L_t' \mcvot, \mcvot).
\end{align*}
Finally, the last claim follows from Proposition \ref{prop_ddtLt}.
\end{proof}

\begin{remark}\label{rem_defh}
For any homogeneous space $(\Go/\Hho,\bar g)$,
the mean curvature vector  $\mcvo_{\bar g}$ and its norm $h(\bar g)$ can be defined as in \eqref{eqn_mcv} and Lemma \ref{lem_h}. 
Recall now that
$\Ricci_{\bar g}=\Riccim_{\bar g}-S^g(\ad(\mcvo_g))$
and that 
$\tr S^g(\ad(\mcvo_g)) =\Vert \mcvo_{\bar g} \Vert_{\bar g}^2=2h(\bar g)$: see the proof of Lemma 3.5 in \cite{BL17}.
Notice moreover, that $\Go$ is unimodular if and only if $h(\bar g)=0$.
\end{remark}

\section{Proof of Theorem \ref{thm_main2}}\label{sec_maxp}

In \cite{Heb} Heber obtained deep structure results for 
 \emph{standard} Einstein solvmanifolds.
  Building on that,  Lauret extended these results to arbitrary Einstein solvmanifolds by showing that they are all standard  \cite{standard}. 
The main result of this section is Theorem \ref{thm_main2}, which generalizes Lauret's result to 
\emph{orbit-Einstein} cohomogeneity-one manifolds: those satisfying the Einstein condition $\Ricci_g(v,v) = c\, g(v,v)$ for all directions 
$v$ tangent to the group orbits.

Let  $(\Go/\Hho,  \bar g)$  be an effective homogeneous space
with compact isotropy $\Hho$ and 
 canonical reductive decomposition $\ggoo =\hgo \oplus \mgo$.  
Extend the $\Ad(\Hho)$-invariant scalar product $\bar g$ on $T_{e\Hho} \Go/\Hho \simeq \mgo$ to an
$\Ad(\Hho)$-invariant scalar product $\bar g$ on $\ggoo$ with $\bar g(\mgo,\hgo)=0$.

\begin{definition}\label{def_stand}
 We call the homogeneous space $ (\Go/\Hho, \bar  g)$ \emph{standard}, if the 
$\bar g$-orthogonal complement of the nilradical $\ngoo$ ---the maximal nilpotent ideal in $\ggoo$--- is a Lie subalgebra of $\ggoo$. 
\end{definition}

Recall that $\kf_\ggoo(\ngoo,\ggoo) = 0$, thus $\ngoo \subset \mgo$ (see \cite[Lemma 5.1]{BL17}), and note
that this definition does not depend on the way the scalar product is extended. Geometrically,
standardness means that the submersion given by the
isometric action of the nilradical $\No\leq \Go$ ($\Lie( \No) =\ngoo$)
on $(\Go/\Hho,  \bar g)$ has
integrable horizontal distribution.


Turning to the proof of Theorem \ref{thm_main2}, recall that we
denoted by $\Hho=\Go_{\gamma(t)}$  the isotropy
 at the points $\gamma(t)$ of a normal, unit-speed geodesic, see Section \ref{sec_cohom1}. 
Consider the endomorphism $\Beta \in \End(\ggoo)$ associated to the homogeneous space $\Go/\Hho$ after fixing the background metric $\bar g:=g_0$, as explained in Appendix \ref{app_beta}. We may write the family $(g_t)_{t\in \RR}$ of scalar products on $\mgo$  induced by $g$ 
as $g_t = q_t \cdot \bar g$,
 for some smooth curve $ (q_t)_{t\in \RR} $ in $\Gmo$ with $q_0 = \Id_\mgo$, see \eqref{eqn_Sym2action}.  By Lemma \ref{lem_lift} and the discussion around \eqref{eqn_QbHtrans} we may even assume that
\begin{eqnarray}\label{eqn_qt}
  (q_t)_{t\in \RR} \subset \QbHmo\,,
\end{eqnarray}
 where
$\QbHmo := \Qb \cap \Gmo$,
 see \eqref{eqn_QbHm},
and $\Qb \subset \Gl(\ggoo)$ is the parabolic subgroup associated to $\Beta$, defined in Appendix \ref{app_beta}.

\begin{definition}\label{def_l}
Let 
\[
	l_\Beta: \RR \to \RR\,\,;\,\,\,t \mapsto 
	\tr \big( L_t \, q_t\,  \Beta^+ q_t^{-1} \big)\,,
\]
where $\Beta^+ := (\tfrac{1}{\Vert \Beta \Vert^2}\Beta +  \Id_\ggo)|_\mg \in \End(\mgo)$.
\end{definition}

By Theorem \ref{thm_beta}, the endomorphism
$\tfrac{1}{\Vert \Beta \Vert^2}\Beta +  \Id_\ggo$ of $\ggoo$
preserves 
$\mgo$ and its kernel contains $\hgo$.
Moreover, the image of the $\bar g$-selfadjoint,  positive-semidefinite endomorphism $\Beta^+$ is precisely $\ngoo$. Thus $\tr \Beta^+\geq 0$.
Notice also that the shape operator 
$L_t$ is symmetric with respect to $g_t$, but $q_t^{-1}L_t q_t$ is symmetric with respect to the background metric $\bar g$.

\begin{lemma}\label{lem_lperiodic}
 If $M^n/\Go = S^1$ then $l_\Beta$ is periodic.
\end{lemma}

\begin{proof}
Assuming that $\gamma(T) \in \Sigma_0=\Go/\Hho$ for $T>0$, 
it is enough to show $l_\Beta(0) = l_\Beta(T)$. Let
$f\in \Go$ with $f\cdot \gamma(0) = \gamma(T)$.
Since $f$ normalizes $\Hho$,  
we have  $q_T \cdot \bar g = (\Ad_\mgo f) \cdot q_0 \cdot \bar g$ by Lemma \ref{lem_Adf}, which implies $q_T^{-1} \cdot (\Ad_\mgo f) \in \Omo$
using  $q_0 = \Id_\mgo$.
On the other hand, after extending $q_T$ to all of $\ggoo$ (as in \eqref{eqn_Gminclusion}), we have that $q_T \in \Qb$ by \eqref{eqn_qt} and \eqref{eqn_QbHm}. 
This, together with $\Ad f \in \Aut(\ggoo) \subset \Qb$, see Theorem \ref{thm_beta}, gives
$B:=q_T^{-1} \cdot (\Ad f)  \in \Or(\ggoo) \cap \Qb = \Kb$.
In particular $B$ commutes with $\Beta$. Thus, 
setting $A:= (\Ad f)|_\mg$  we obtain by Corollary \ref{cor_shape}
\begin{align*}
	l_\Beta(T) =  \tr (L_T \, q_T\,  \Beta^+  q_T^{-1}) 
	=  \tr (A L_0  A^{-1} q_T \Beta^+ q_T^{-1} )
		= \tr (L_0 B^{-1} \Beta^+ B )  = l_\Beta(0)\,,
\end{align*}
concluding the proof.
\end{proof}

We turn now to an apriori estimate for $h$, which holds for arbitrary homogeneous spaces:

\begin{lemma}\label{lem_ricHestimate}
For all $t\in \RR$ we have $- \ricci_t(\mcvot, \mcvot)\cdot\tr \Beta^+  \geq 4 \,h^2(t) $.
\end{lemma}

\begin{proof}
By  \cite{AL16} we have
\[
	\ricci_t (\mcvot, \mcvot) = - \Vert S_t(\ad_\mgo \mcvot)\Vert_t^2.
\]
Here, $\ad_\mgo (\mcvot) = \proj_\mgo \circ  \ad (\mcvot) |_\mgo \in \End(\mgo)$, $\proj_\mgo$ denotes the projection with respect to the decomposition $\ggoo = \hgo \oplus \mgo$, and for an
endomorphism $E\in \End(\mgo)$ we denote its $g_t$-symmetric part by
$S_t(E) = \unm(E+E^T)$ and its $g_t$-norm squared by
$\Vert E \Vert_t^2 = \la E, E \ra_t := \tr (E E^T)$,
transposes taken with respect to the metric $g_t$.
By Cauchy-Schwarz inequality
 we have
\[
	\Vert S_t(\ad_{\mgo} (\mcv_t))\Vert_t^2 \, 
\cdot	
	\,  \Vert  q_t \Beta^+ q_t^{-1} \Vert_t^2 \geq \big\la  S_t(\ad_{\mgo} (\mcv_t)), q_t \Beta^+ q_t^{-1} \big\ra_t^2.
\]
Using that $q_t \Beta^+ q_t^{-1}$ and $S_t(\ad_{\mgo} (\mcv_t) )$ are $g_t$-symmetric, we may rewrite the latter as
\[
	-\ricci_t (\mcvot, \mcvot) 
	\, \cdot \, 
	 \tr ((\Beta^+)^2 )
	\, \geq  \,
	\big( \tr  (\ad_{\mgo} (\mcv_t) q_t \Beta^+ q_t^{-1}  ) \big)^2.
\]
A short computation using that $\tr \Beta = -1$ shows now
 $\tr ((\Beta^+)^2) = \tr \Beta^+$. On the other hand, Theorem \ref{thm_beta} yields $\ad (\mcvot) \in \slgb$, and since $\slgb$ is an ideal in $\qgb$ we  have  $q_t^{-1} \ad (\mcv_t) q_t \in \slgb$. Since $\tr (E_\Beta \cdot\Beta)=0$ for all
$E_\beta \in \slgb$ this implies by the very definition of $\Beta^+$ that
\[
	\tr_\mgo \big( q_t^{-1} \ad_\mgo (\mcv_t) q_t \Beta^+ \big) 
	= \tr_\ggoo \big( q_t^{-1} \ad (\mcv_t) q_t \Beta^+ \big) 
	= \tr \ad_{\mgo} (\mcv_t) = 2 \, h(t)
\]
using \eqref{eqn_mcv} and $[\mgo,\hgo]\subset \mgo$. 
This shows the claim.
\end{proof}


\begin{lemma}[Maximum principle]\label{lem_betapgeqh}
Let  $(M^n,g)$ be orbit-Einstein with Einstein constant $-1$ and 
suppose that $M^n/\Go = S^1$. Then, 
$\tr \Beta^+ \geq 2 \, h(t)$ for all $t\in \RR$.  
\end{lemma}

\begin{proof}
Since $\Beta^+ \geq 0$ by Theorem \ref{thm_beta},
 $h(t)\equiv 0 \leq \tr \Beta^+$ if $\Go$ is unimodular:
see \mbox{Remark \ref{rem_defh}}.
If $\Go$ is not unimodular, then  $h, \tr \Beta^+>0$. 
Since $M^n/\Go = S^1$, $h$ is periodic, since it is defined geometrically:
see also Remark \ref{rem_defh}. Thus
 it attains a global maximum $h(t_0) > 0$, $t_0\in [0,T]$. 
We deduce from Lemma \ref{lem_h} and Lemma \ref{lem_ricHestimate} that
\[
	0 \, \geq \,  2 \, \Vert L_{t_0} \mcv_{t_0}\Vert_{t_0}^2  - 2\, h (t_0) - \ricci_{t_0} (\mcv_{t_0}, \mcv_{t_0}) 
	\geq 
	2 \, h(t_0) \,  \left( \tfrac{2 h(t_0)}{ \tr \Beta^+} - 1  \right).
\]
Hence, $\tr \Beta^+ \geq 2 \, h(t_0) \geq  2 \, h(t)$ for all $t\in \RR$, as claimed.
\end{proof}


\begin{lemma}[Lyapunov function]\label{lem_lmonotone}
Let $(M^n,g)$ be orbit-Einstein with Einstein constant $-1$ and
suppose that $M^n/\Go = S^1$.
 Then 
\[
	w_\Beta:\RR \to \RR\,\,;\,\,\,t\mapsto	l_\Beta(t)\cdot e^{\int_0^t \tr L_s ds}
\]
is non-decreasing. If $w_\Beta$ is constant, then $\Beta^+ \in \Der(\ggoo)$ and $[q_t, \Beta] = 0$ for all $t\in \RR$.
\end{lemma}

\begin{proof}
Thanks to Proposition \ref{prop_ddtLt} and Definition \ref{def_l}
we have
\begin{align*}
	l_\Beta'(t) &= \tr \big( L'_t \, q_t \, \Beta^+ q^{-1}_t\big) 
	+ \tr \big( L_t \big[  q'_tq^{-1}_t, q_t \Beta^+ q^{-1}_t \big]   \big) \\
	&= \tr \big( (\Ricci_t + \Id_\mg) \, q_t \, \Beta^+ q^{-1}_t\big) 
	+ \tr \big(  (q^{-1}_t L_t q_t) \, \big[  q^{-1}_t q'_t, \Beta \big] \big)
	-(\tr L_t)\cdot l_\Beta(t)\,.
\end{align*}
Regarding the second term, first notice that for $X,Y \in \mgo$ we have
by Lemma \ref{lem_ddtgt} 
\[
	2g_t(L_tX,Y)=g'_t(X,Y) = (q_t \cdot \bar g)'(X,Y) = 
	- g_t\big( (q'_t q^{-1}_t) X ,Y \big)
	 - g_t \big( X, (q'_t q^{-1}_t) Y \big)\,.
\]
Thus, $L_t + q_t' q_t^{-1} \in \sog(\mgo, g_t)$ or equivalently
$q_t^{-1} (L_t + q_t' q_t^{-1}) q_t \in \sog(\mgo,\bar g)$ for all
 $t\in \RR$, since $g_t = q_t \cdot \bar g$. Since 
 $q_t^{-1} L_t q_t$ is $\bar g$-symmetric for each $t$ and 
  $q_t^{-1} q_t' \in \qgb=T_e \Qb$ by \eqref{eqn_qt} and \eqref{eqn_QbHm}, after extending these linear maps trivially to $\glg(\ggoo)$,
 we are in position to apply Lemma \ref{lem_LQbeta} for
  $S = q_t^{-1} L_t q_t$ and  $Q = q_t^{-1} q_t'$. This yields
\begin{equation}\label{eqn_secondterm}
	\tr \big(  (q^{-1}_t L_t q_t) \, \big[  q^{-1}_t q'_t, \Beta \big] \big) \geq 0\,,
\end{equation}
with equality if and only if $[q_t^{-1} q_t', \Beta] = 0$.

Regarding the first term, notice that by Theorem \ref{thm_beta} we have
\begin{equation}\label{eqn_ricbetageq0}
	\tr (\Ricci_t q_t \beta^+ q_t^{-1}) \geq - 2 \, h(t),
\end{equation}
 with equality if and only if 
 $q_t (\tfrac{1}{\Vert \Beta \Vert^2}\Beta +  \Id_\ggo) q_t^{-1} \in \Der(\ggoo)$. 
  Thus, 
 \[
 	\tr \big( (\Ricci_t + \Id_\mg) \, q_t \, \Beta^+ q^{-1}_t\big) \geq  -2 \, h(t) + \tr \Beta^+  \geq 0
 \]
by Lemma \ref{lem_betapgeqh}.
We conclude that $l_\Beta'(t)+l_\Beta(t) \tr L_t \geq 0$, showing the first claim.

Finally, suppose $w_\Beta'\equiv 0$ and set $C_t:=[q_t, \Beta] $ for all $t\in\RR$. Then from equality in \eqref{eqn_secondterm} we deduce
$C_t'=C_t \cdot (q_t^{-1}q_t')$ with $C_0=0$, since $q_0 = \Id_\mgo$. 
This shows $C_t \equiv 0$.
From the equality condition in \eqref{eqn_ricbetageq0} we then obtain $\Beta^+ \in \Der(\ggoo)$.
\end{proof}

\begin{proof}[Proof of Theorem \ref{thm_main2}]
We have $\int_0^T \tr L_s \,ds =0$,
since by assumption the $\Go$-orbits are integrally minimal. As a consequence, by Lemma \ref{lem_lperiodic} and Lemma \ref{lem_lmonotone}, the function $w_\Beta(t)$ is monotone and periodic, therefore constant. The rigidity conditions in Lemma \ref{lem_lmonotone} imply that $\beta^+_\ggoo:=
\tfrac{1}{\Vert \Beta \Vert^2}\Beta +  \Id_\ggo \in \Der(\ggo)$ and $[q_t, \Beta] = 0$ for all $t\in \RR$. Using that $0 = [q_t, \Beta] = [q_t,\beta^+_\ggoo]$ and 
$\ngoo = \Im \Beta^+= \Im \Beta^+_\ggoo$, 
we see that with respect to $g_t = q_t \cdot \bar g$, the $g_t$-orthogonal complement of $\ngoo$ in $\ggoo$ is a fixed subspace $\ngoo^\perp$ of $\ggoo$, for all $t\in \RR$.  Since $\Beta^+_\ggoo$ 
is $g_0$-symmetric (and in fact also $g_t$-symmetric for all $t\in\RR$), $\ngoo^\perp = \ker \Beta^+_\ggoo$. But  
the kernel of a derivation is always a Lie subalgebra, therefore all orbits are standard homogeneous spaces as claimed.
\end{proof}

\begin{remark}\label{rem_intmin}
When writing $g_t ( \,\cdot \, ,\, \cdot )
	=2\, \bar g(G_t \,\cdot \, ,\, \cdot )$ for a 
	$\Go$-invariant background metric $\bar g$ on $\Go/\Hho$,
	 integral minimality is equivalent to the periodicity of
	the volume $V(t)=\det G_t$ because $\tr L_t =V'(t)/V(t)$.
Periodicity of $V$ is automatically satisfied if $\Go$ is unimodular, but a true  assumption if $\Go$ is not: one may declare any smooth curve $\gamma:[0,T] \to M^n$ intersecting all orbits transversally to a normal geodesic if $\gamma'(T)=(dL_{\bar g})\cdot \gamma'(0)$, see \cite{GWZ}.
\end{remark}

\section{Euclidean homogeneous spaces}\label{sec_homspaces}

The Euclidean space $\RR^n$ can be presented as an effective homogeneous space $\G/\Hh$ in many different ways. Here and in what follows, $\G$ is a non-compact,
connected  Lie group and $\Hh\leq \G$ a connected, compact subgroup.
Recall that $\Hh$ is compact if and only if 
$\G$ is closed in $\Iso(\G/\Hh,g)$ for some (and also any) $\G$-invariant
metric $g$: see \cite[Thm.1.1]{Dtt88}. For instance, $\RR^n$ is itself an abelian Lie group, and more generally a solvmanifold, that is
 $\RR^n=\G/\{e \}$ for any simply-connected, solvable 
Lie group $\G$ of dimension $n$.  
Any non-compact irreducible symmetric space 
has a presentation  $\RR^n=\Ll/\K$ with $\Ll$
simple and centerless, and $\K\leq \Ll$  maximal compact.
Another possibility is that $\G$ is the universal covering group of $\Sl(2,\RR)$ and $\Hh = \{e\}$, or a product of any of the above examples.

The main aim of this section is to characterize certain effective presentations $\RR^n=\G/\Hh$,  where $\dim \G$ has minimal dimension. Let $\ggo = \lgo \ltimes \rg$ be a  \emph{Levi decomposition} 
of  $\Lie(\G)=\ggo$, by which we mean that $\lgo$ is a maximal semisimple Lie subalgebra and $\rg$ the radical of 
$\ggo$ (maximal solvable ideal). Let 
 $\Ll, \Rr \leq \G$ denote the connected Lie sugbroups of $\G$
 with $\Lie(\Ll)=\lgo$ and $\Lie(\Rr)=\rg$, and
suppose that $\Lie(\Hh)=\hg \subset \lgo$.   
Notice that for non-compact homogeneous Einstein spaces, the structure results 
yield such a Levi decomposition, see Section \ref{sec_proofmain}.

\begin{proposition}\label{prop_Rnhomognew}
Let $(\RR^n,g)$ be a homogeneous Euclidean space with effective presentation $\RR^n=\G/\Hh$, $\G$ connected, $\Hh$ compact and $\dim \G$ minimal. Assume in addition that there exists a Levi decomposition $\ggo = \lgo \ltimes \rg$ with $\hg \subset \lgo$. Then, $\G \simeq \Ll \ltimes \Rr$, $\Hh$ is semisimple, and $\Ll/\Hh$ is  an $\RR^r$-bundle over a Hermitian symmetric space of 
non-compact type.  
\end{proposition}


A \emph{Cartan decomposition} $\lgo = \kg \oplus \pg$ for the semisimple Lie algebra $\lgo$ is obtained by letting $\kg$ (resp.~$\pg$) be the $+1$ (resp.~$-1$) eigenspace of a Cartan involution $\sigma:\lgo \to \lgo$, an involutive automorphism of $\lgo$ such that $-\kf_\lgo(\cdot, \sigma\cdot) > 0$, where $\kf_\lgo$ denotes the Killing form of $\lgo$. The pair $(\kg, \pg)$ satisfies the bracket relations $[\kg,\kg] \subset \kg$, $[\kg,\pg] \subset \pg$, $[\pg,\pg] \subset \kg$. If $\lgo$ has no compact simple ideals, we call $(\lgo,\kg)$  a \emph{symmetric pair of non-compact type}. 
The subalgebra $\kg$ will be called a \emph{maximal compactly embedded subalgebra}. 
Notice however, that the corresponding connected Lie subgroup $\K\leq \Ll$ is compact if and only if $Z(\Ll)$ is finite  \cite[Ch.6, Thm.1.1]{Helgason01}.

A special class of symmetric spaces
are the \emph{Hermitian symmetric spaces},
see \cite[Ch.~VIII, Thm.6.1]{Helgason01} or \cite[7.104]{Bss}. Irreducible ones correspond to irreducible symmetric pairs $( \lgo, \kg)$  with $\dim \zg( \kg) = 1$. Writing $ \kg=\zg( \kg)\oplus
 \kg_{ss}$, we obtain a homogeneous space
$ \Ll/ \K_{ss}$, which is an $S^1$-bundle over the
Hermitian symmetric space $ \Ll/ \K=\RR^k$ (assuming $\Ll$ has finite center). As a consequence, the universal covering space of 
$ \Ll/ \K_{ss}$, still being a
homogeneous space, is diffeomorphic to $\RR^{k+1}$ and on
Lie algebra level still presented by the pair $( \lgo, \kg_{ss})$.
As a smooth manifold it is an $\RR$-bundle over $\Ll/\K$. Considering
products of such $\RR$-bundles  we obtain
more complicated presentations $\Ll/\Hh=\RR^n$, which
on Lie algebra level can be described as follows:
\begin{eqnarray}\label{eqn_RbundleHSS}
	\lgo:= \lgo_1 \oplus \cdots \oplus \lgo_r\,\, ,\quad
	\hg:= \kg_1^{ss} \oplus \cdots \oplus \kg_r^{ss} \oplus \zg(\hg)\,\,,\quad
	\zg(\hg) \subset  \zg(\kg_1 \oplus \cdots \oplus \kg_r)
\end{eqnarray}
where  $(\lgo_i, \kg_i = \zg(\kg_i)\oplus \kg_i^{ss})$
is an irreducible Hermitian symmetric pair of non-compact type, $1 \leq i \leq r$, and $\zg(\hg)$ is a proper subalgebra of $ \zg(\kg_1 \oplus \cdots \oplus \kg_r)$. Now $\hg$ is semisimple if and only if
$\zg(\hg)=0$, in which case 
the homogeneous space $\Ll/\Hh$
corresponding to such pairs $(\lgo,\hg)$ will be called an 
\emph{$\RR^r$-bundle over a non-compact Hermitian symmetric space}.

\begin{remark}\label{rem_euclhom}
Irreducible Hermitian
symmetric spaces $\Ll/\K$  of non-compact type
are given by
 $(\slg(2,\RR),\RR)$,
 $(\sug(p,q), \RR \oplus \sug(p) \oplus \sug(q))$,
 $ p\geq q \geq 1, (p,q)\neq (1,1),(2,2)$, 
 $(\sog^*(2n), \RR \oplus \sug(n))$, $ n \geq 5$,
$(\spg(n,\RR),\RR \oplus \sug(n))$, $ n\geq 2$,
 $(\eg_6, \RR \oplus \sog(10))$, $(\eg_7, \RR \oplus \eg_6) $
and $ (\sog(m,2), \RR \oplus \sog(m))$, $m\geq 3  $,  called hyperquadrics: see  \cite{BB}[p.569]. 

We now describe the space $\mathcal{M}_{\Ll/\K_{ss}}$ of $\Ll$-invariant
 metrics on $\Ll/\K_{ss}$. We will show that for any $g \in 
 \mathcal{M}_{\Ll/\K_{ss}}$ there exists a
Cartan decomposition $\lgo = \kg \oplus \pg$ with
 $\kg \cdot o \perp_g \pg \cdot o$, $o = e\K_{ss}$. 

 Let $\mg=\zg(\kg) \oplus \pg$ be the
 canonicial complement of $\hg=[\kg,\kg]$ in $\lgo$, where
 $\kg=\zg(\kg)\oplus \kg_{ss}$.
 Except for the first and the last cases above, $\pg$ is  $\Ad(\K_{ss})$-irreducible  and non-trivial,
 see \cite {BB} and \cite{WZ86}, and consequently
 $\dim \mathcal{M}_{\Ll/\K_{ss}}=2$ and
 $\kg \cdot o \perp_g \pg \cdot o$.
In the first case, $\mg=\slg(2,\RR)$ is a trivial
$\Ad(\K_{ss})$-module, since $\K_{ss}=\{e\}$. It follows  that
$\dim \mathcal{M}_{\Ll/\K_{ss}}=6$. However, for each 
$g$ there exists a  ``Milnor basis'' \cite{Mln}
of $\slg(2,\RR)$ diagonalising both the metric and the bracket at the same time. Consequently,
for any of the three $\kg=\sog(2)$ adapted to this basis
we have $\kg \cdot o \perp_g \pg \cdot o$.
Finally, for the hyperquadrics, the isotypical summand
$\pg$ is the sum of two $\Ad(\K_{ss})$-irreducible, equivalent, non-trivial
representations of real type: see  \cite {BB} and \cite{WZ86}.
Thus $\dim \mathcal{M}_{\Ll/\K_{ss}}=4$ and again we have
 $\kg \cdot o \perp_g \pg \cdot o$.

\end{remark}

The proof of Proposition \ref{prop_Rnhomognew} will be clear after the following series of lemmas.

\begin{lemma}\label{lem_Leviglobal}
Under the assumptions of Proposition \ref{prop_Rnhomognew}, we have that $\Hh \leq \Ll$ and $\G \simeq \Ll \ltimes \Rr$.
\end{lemma}

\begin{proof}
The universal cover $\tilde \G$ of $\G$ acts transitively and almost-effectively on $(\RR^n,g)$, with connected isotropy $\tilde \Hh$ at $0$, since $\RR^n$ is simply-connected. 
Thus, if $\tilde \Ll, \tilde \Rr \leq \tilde \G$ denote  the connected Lie subgroups with Lie algebras $\lgo, \rg$, respectively, 
then  $\hg \subset \lgo$ implies  $\tilde \Hh \leq \tilde \Ll$. Moreover, since $\tilde \G$ is simply-connected, $\tilde \G \simeq \tilde \Ll \ltimes \tilde \Rr$ \cite[3.18.13]{Varad84}. 
We have $\G = \tilde \G / \ker \tau$, where $\tau : \tilde \G \to \Iso(\RR^n,g)$ is the morphism defining the action. Clearly $\ker \tau \leq \tilde \Hh$, thus $\G \simeq \Ll \ltimes \Rr$, where $\Rr \simeq \tilde \Rr$ is simply-connected and $\Ll \simeq \tilde \Ll / \ker \tau$.
\end{proof}

 
\begin{lemma}\label{lem_maxab}
Under the assumptions of Proposition \ref{prop_Rnhomognew}, there exists a Cartan decomposition 
$\lgo =\kg \oplus \pg$ such that $\hg \subset \kg$.
Moreover,
 $\kg = [\kg,\kg]\oplus \zg(\kg)$ is an abelian extension of $\hg$, that is, 
 $\kg_{ss}:=[\kg,\kg] \subset \hg$, and $(\lgo,\kg)$ is a symmetric  pair of non-compact type.
\end{lemma}

\begin{proof}
Since $\Ll$ is semisimple, $\Ad:\Ll \to \Gl(\lgo)$
has discrete kernel and 
 $\Lie(\Ad(\Ll))=\lgo$. 
Moreover, the compact subgroup $\Ad(\Hh)$
is contained in a maximal compact subgroup $\K$ of
the centerless semisimple Lie group $\Ad(\Ll)$.  
By \cite[Ch.6, Thm.1.1]{Helgason01}, $\kg := \Lie(\K)$
is the fixed point set of a Cartan involution. The corresponding Cartan decomposition  $\lgo = \kg \oplus \pg$ satisfies $\hg \subset \kg$.

We now show that $[\kg,\kg] \subset \hg$. Let $\K, \K_{ss} := [\K,\K]$ be the connected Lie subgroups of $\G$ with Lie algebras $\kg, \kg_{ss}$, respectively. It is well-known that $\G/\Hh = \RR^n$ if and only if $\Hh$ is a maximal compact subgroup of $\G$, see for instance \cite[Prop.3.1]{alek}. The group $\K_{ss}$ is semisimple and has a compactly embedded Lie algebra $\kg_{ss} \subset \kg$, therefore it is compact. By uniqueness of maximal compact subgroups, there exists $x\in \G$ such that $x \K_{ss} x^{-1} \subset \Hh$. At Lie algebra level this implies that $\Ad(x)(\kg_{ss}) \subset \hg \subset \kg$. But the Lie algebra $\kg = \kg_{ss} \oplus \zg(\kg)$ has a unique maximal semisimple ideal, therefore we must have $\kg_{ss} = \Ad(x) (\kg_{ss})$ and  $\kg_{ss} \subset \hg$ as claimed.

To prove the last claim let $\lgo_{c}$ (resp.~ $\lgo_{nc}$) be the sum of all simple ideales of $\lgo$ of compact (resp.~non compact) type, and let $\Ll_c, \Ll_{nc} \leq \Ll$ be the corresponding connected Lie subgroups. The ideal $\lgo_c$ is the unique maximal semisimple ideal of compact type in $\lgo$.  Arguing as above, since $\Ll_{c}$ is compact, there exists $x\in \G$ such that $x \Ll_c x^{-1 } \leq \Hh$, thus $\Ad(x) (\lgo_c) \subset \hg \subset \lgo$. Then $\lgo_c = \Ad(x) (\lgo_c)$ by uniqueness, and $\lgo_c \subset \hg$. It then follows easily that the connected subgroup of $\G$ with Lie algebra $\lgo_{nc} \ltimes \rg$ is closed in $\G$ by Lemmas \ref{lem_Leviglobal} and \ref{lem_Sfactor}, and acts transitively on $\G/\Hh$. This contradicts minimality of $\dim \G$, unless $\lgo_c = 0$, and the conclusion follows. 
\end{proof}

In order to prove the last part of Proposition \ref{prop_Rnhomognew} let us introduce the following class of subgroups of $\G$, which will be very important in the sequel. 

\begin{definition}\label{def_Iwasawa}
Let $\G/\Hh$ be a homogeneous space. Fix a Levi decomposition $\ggo = \lgo \ltimes \rg$ with $\hg \subset \lgo= \lgo_1 \oplus \cdots \oplus \lgo_r$, written as a sum of simple ideals, and consider an Iwasawa decomposition $\lgo_1 = \kg_1 \oplus \ag_1 \oplus \ngo_1$. 
 The connected Lie subgroup 
$\Go \leq \G$ with Lie algebra $\ggoo := ((\ag_1 \oplus \ngo_1) \oplus \lgo_2 \oplus \cdots \oplus \lgo_r) \ltimes \rg$ is called an \emph{Iwasawa subgroup associated to $(\lgo_1, \kg_1)$}. 
\end{definition}


\begin{lemma}\label{lem_Iwaclosed}
For a Lie group $\G \simeq \Ll \ltimes \Rr$ with a gobal Levi decomposition, Iwasawa subgroups are closed.
\end{lemma}

\begin{proof}
By the assumption, it is enough to show that the connected Lie subgroup $\bar \Ll$ of $\Ll$ with Lie algebra $\bar \lgo := (\ag_1 \oplus \ngo_1) \oplus \lgo_2 \oplus \cdots \oplus \lgo_r$ is closed in $\Ll$. Let $\Ll_i \leq \Ll$ denote the connected Lie subgroup with Lie algebra $\lgo_i$. Since the Iwasawa decomposition holds at the group level, we have $\Ll_1 = \K_1 \Aa_1 \N_1$ diffeomorphic to $\K_1 \times \Aa_1 \times \N_1$. In particular, $\Aa_1 \N_1$ is closed in $\Ll_1$. Recall that for a connected semisimple Lie group, $\Ll_i$ is closed in $\Ll$ for all $i=1,\ldots, r$, by Lemma \ref{lem_Sfactor}. Then the lemma follows from the fact that $\bar \Ll = \Aa_1 \N_1 \Ll_2 \cdots \Ll_r$.
\end{proof}

\begin{lemma}\label{lem_Rnhomog}
Under the assumptions of Proposition \ref{prop_Rnhomognew}, if $\lgo = \kg \oplus \pg$ is a Cartan decomposition with $\hg \subset \kg$, then $\hg = \kg_{ss}$.
\end{lemma}

\begin{proof}
Thanks to Lemma \ref{lem_maxab}, we only need to prove that $\zg(\hg) = 0$. Assume on the contrary that  $\zg(\hg)$ projects nontrivially onto 
$\zg(\kg_1)$. This means that there exists $Z \in \zg(\hg)$ such that 
\begin{equation}\label{eqn_Z1}
	Z = Z_1 + Z_2, \qquad Z_1 \in \zg(\kg_1) \setminus \{0 \}, \, \, Z_2 \in \zg(\kg_2 \oplus \cdots \oplus \kg_r)
\end{equation}
in the notation of \eqref{eqn_RbundleHSS}. Let $\Go \leq \G$ be an Iwasawa subgroup associated to $(\lgo_1, \kg_1)$. We claim that 
$ \ggo=\hg + \bar \ggo $, which would follow from $\kg_1 = \kg_1^{ss} \oplus \zg(\kg_1) \subset \hg + \bar \ggo$. By Lemma \ref{lem_maxab} it suffices to show that $\zg(\kg_1) \subset \hg+\bar \ggo$,
which follows from \eqref{eqn_Z1}, since
$\dim(\zg(\kg_1))=1$.

Now $\bar \G$ acts transitively on $\G/\Hh$ by the above claim, effectively because $\G$ acts effectively, and with compact isotropy, since this is equivalent to $\bar \G$ being closed in the isometry group of $g$, which holds thanks to Lemma \ref{lem_Iwaclosed} and the fact that $\G$ is closed in $\Iso(\RR^n,g)$. But $\dim \bar \G < \dim \G$, and this contradicts the minimality of $\dim \G$.   
\end{proof}

\section{Periodicity}\label{sec_period}

In this section we show that after passing to a  quotient we may
assume that Iwasawa subgroups 
act on $\RR^n=\G/\Hh$ with cohomogeneity one and 
orbit space $S^1$. The most basic example for periodicity is given by
the universal cover  
$\tilde \G\simeq \RR^3$ of $\G=\Sl(2,\RR)$. 
Since $\Sl(2,\RR) = \Gamma \backslash \tilde \G$ for some discrete subgroup $\Gamma \leq Z(\tilde \G)$, 
for any left-invariant metric $\tilde g$ on $\tilde \G$ 
there is a left-invariant metric $g$ on $\Sl(2,\RR)$ locally isometric 
$\tilde g$. Now from \emph{any} Iwasawa decomposition $\Sl(2,\RR) = \K\Aa \N$ one gets a solvable subgroup $\Go=\Aa \N \leq \Sl(2,\RR)$ acting on $(M^3=\Sl(2,\RR), g)$ with cohomogeneity one and orbit space $\K \simeq \SO(2) \simeq S^1$.

\begin{proposition}[Perodicity]\label{prop_periodicity}
Let $(\RR^n=\G/\Hh, g)$ be an effective presentation, $\Hh$ compact and $\dim \G$ minimal, and let $\ggo = \lgo \ltimes \rg$ be a Levi decomposition with $\hg \subset \lgo$. Then, there exist a discrete, central subgroup $\Gamma \leq \G$ for which $(M^n=\G_\Gamma/\Hh_\Gamma, g_\Gamma) := (\G/\Hh, g)/\Gamma$, $\G_\Gamma = \G/\Gamma$, $\Hh_\Gamma = \Hh/(\Hh \cap \Gamma)$, is a homogeneous space locally isometric to $(\RR^n,g)$ such that \emph{any} Iwasawa subgroup $ \Go \leq \G_\Gamma$ acts on $M^n$ with cohomogeneity one, embedded orbits and orbit space $S^1$.
\end{proposition}

\begin{proof}
By Proposition \ref{prop_Rnhomognew} we have a global Levi decomposition $\G \simeq \Ll \ltimes \Rr$. Let $\Gamma := Z(\G) \cap \Ll$, a central subgroup of $\G$ and $\Ll$, hence discrete by semisimplicity. It is not hard to see that the group $\G_\Gamma := \G/\Gamma$ acts transitively on the double coset $\Gamma \backslash \G / \Hh$, with isotropy $\Hh_\Gamma := \Hh / (\Hh \cap \Gamma)$. Since $\Gamma \leq \Ll$ we clearly also have a global Levi decomposition $\G_\Gamma \simeq \Ll_\Gamma \ltimes \Rr$, $\Ll_\Gamma := \Ll/\Gamma$. Since $\G_\Gamma$ and $\Hh_\Gamma$ are connected
with $\Lie(\G_\Gamma)=\ggo$ and $\Lie(\Hh_\Gamma)=\hg$,
the space of homogeneous metrics on $\G/\Hh$ and
$M^n=\G_\Gamma/\Hh_\Gamma$ agrees. Moreover,
Iwasawa subgroups $\Go_\Gamma \leq \G_\Gamma$ will act with cohomogeneity one on $M^n$, and
by Lemma \ref{lem_Iwaclosed} they are closed in $\G_\Gamma$.

It remains to be shown that the orbit space is $S^1$. To that end, if $\Go_\Gamma \leq \G_\Gamma$ is an Iwasawa subgroup associated to $(\lgo_1, \kg_1)$, let $\K_1 \leq \Ll_1 \leq \Ll_\Gamma$ be the corresponding connected Lie subgroups with Lie algebras $\kg_1 \subset \lgo_1$, respectively. By Proposition \ref{prop_Rnhomognew} (see also \eqref{eqn_RbundleHSS}) we have $\kg_1 = \kg_1^{ss} \oplus \zg(\kg_1)$ with $\dim \zg(\kg_1) = 1$. Consider now the action of the connected Lie group $Z(\K_1) \subset N_{\G_\Gamma}(\Hh_\Gamma)$ 
 on $M^n$ by right-multiplication. Since this action commutes with the action of $\Go_\Gamma$ by left-multiplication, the direct product 
$\Go_\Gamma \times Z(\K_1)$ acts 
transivtiely on $M^n$. 
Thus, $Z(\K_1)$ acts transitively on the orbit space $M^n/\Go_\Gamma$. 
Finally notice, that the group $\Gamma$ is the kernel of the adjoint  representation $\Ad(\G):\G\to \Aut(\G)$ restricted to $\Ll$. 
Since  $\Ad : \Ll_\Gamma \to 
\End(\ggo)$ is a faithful  representation, 
by \cite[Prop. 16.1.7]{HlgNeb}
the group $\Ll_\Gamma$ 
is linearly real reductive, implying finite center. By the same
reason the center of $\Ll_1 \leq \Ll_\Gamma$ is also finite, thus
$\K_1$ is compact by \cite[Ch.6, Thm.1.1]{Helgason01}
 and therefore $Z(\K_1) \simeq S^1$.
\end{proof}

\section{Integral minimality}\label{sec_integralm}

Let $(\RR^n = \G/\Hh, g)$ satisfy the assumptions of Proposition \ref{prop_periodicity} and consider its quotient space  $(M^n=\G_\Gamma/\Hh_\Gamma, g_\Gamma)$ given by that result. The proof of Proposition \ref{prop_periodicity} yields  also a global Levi decomposition 
$\G_\Gamma = \Ll \ltimes \Rr$ with $\Hh_\Gamma \leq \Ll$. From now on we also assume that at $o:= e\Hh$ we have 
\begin{equation}
	\Ll \cdot o \perp_g  \Rr \cdot o\,. \label{eqn_orth}
\end{equation}
To ease notation, in what follows we will drop the subscripts $\Gamma$ and write simply $\G/\Hh$. 

\begin{definition}\label{def_Levi}
 A homogeneous space $(\G/\Hh, g)$ as above satisfying \eqref{eqn_orth} and the conclusion of Proposition \ref{prop_periodicity} is said to have a \emph{Levi presentation}. 
\end{definition}


The main result in this section is the following

\begin{theorem}[Integral minimality]\label{thm_minimal}
Let $(\G/\Hh,g)$  have a Levi presentation and choose a simple factor $\lgo_1\subset \lgo$ and a maximal compactly embedded subalgebra $\kg_1 \subset \lgo_1$. Then, there exists an Iwasawa subgroup $\bar \G \leq \G$ associated to $(\lgo_1,\kg_1)$ which acts on $\G/\Hh$ with cohomogeneity one, embedded orbits and orbit space $S^1$. Moreover,
\begin{equation}\label{eqn_intmin}
	\int_{S^1} \tr L_t \, \,  dt = 0\,,
\end{equation}
where  $\tr L_t$ denotes the mean curvature of the $\bar \G$-orbits.
If $\lgo_1 \simeq \slg(2,\RR)$, then the orbits of \emph{any} Iwasawa subgroup associated to $(\lgo_1,\kg_1)$ are in addition minimal hypersurfaces.
\end{theorem}

In the case of hyperquadrics ($\lgo_1 \simeq \sog(m,2)$, see Remark \ref{rem_euclhom}) it can be shown that the $\bar \G$-orbits are not minimal hypersurfaces for a generic invariant metric $g$.

The idea is to reduce the proof to the case $\G$ simple. To that end,  fix a simple factor $\lgo_1 \subset \lgo$ and set $\lgo_{\geq 2} := \lgo_2 \oplus \cdots \oplus \lgo_r$, so that $\lgo = \lgo_1 \oplus \lgo_{\geq 2}$ and  $\ggo_2 := \lgo_{\geq 2} \ltimes \rg \, \, \subset \,\, \ggo$
is an ideal. Denote respectively by $\Ll_{\geq 2}, \G_2\leq \G$ the connected Lie subgroups with Lie algebras $\lgo_{\geq 2}, \ggo_2$. Since $\G \simeq \Ll\ltimes \Rr$, we have that $\G_2 \simeq \Ll_{\geq2} \ltimes \Rr$, with $\Ll_{\geq 2}$ closed in $\Ll$ by Lemma \ref{lem_Sfactor}. Thus, $\G_2$ is a closed normal subgroup of $\G$.  Consider the isometric action of $\G_2$ on $\G/\Hh$ by left-multiplication. The closed subgroup $\G_2$  acts properly on $\G/\Hh$, and being normal in $\G$, all isotropy subgroups for this action are conjugate. Thus,
the orbit space $\G_2 \backslash \G/\Hh$ is a smooth manifold. It admits a transitive action by $\G_2 \backslash \G =: \Ll_1$ with 
compact isotropy $\Hh_1 \simeq \Hh/\Hh\cap \G_2$: see
 \cite[Prop. 5]{BB}.

Let $\bar g$ be the unique $\Ll_1$-invariant metric on $\Ll_1/\Hh_1$ such that we have a Riemannian submersion
\begin{equation}\label{eqn_Rsubmersion}
	\pi : (\G/\Hh,g ) \, \to \,  (\Ll_1/\Hh_1, \bar g).
\end{equation}

\begin{lemma}\label{lem_Rmin}
The fibres of $\pi$ are pairwise isometric, minimal submanifolds of 
$(\G/\Hh, g)$.
\end{lemma}

\begin{proof}
Notice that for any $x\in \G$ we have
\[
	\G_2 \cdot (x\Hh) = x \cdot (x^{-1} \G_2 x ) \cdot o = x \cdot (\G_2 \cdot o)\,.
\] 
Thus any two $\G_2$-orbits are isometric via an ambient isometry. 
By Lemma \ref{lem_Rnhomog}  $\hg = \hg_1 \oplus \hg_{2}$ is a sum of ideals, with $\hg_{2} \subset \ggo_2$. Then the isotropy subalgebra of the homogeneous space $\G_2 \cdot o \simeq \G_2 / (\Hh \cap \G_2)$ is $\hg \cap \ggo_2 = \hg_2$. 
Denote by $\lgo_1 = \hg_1 \oplus \mg_1$, $\ggo_2 =\hg_2 \oplus \mg_2$ the corresponding canonical reductive decompositions. 
Let $\{ U_i\}$ be Killing fields in $\mg_2 \subset \ggo_2$ and let $\{ N_j\}$ span the normal bundle of $\G_2 \cdot o$ close to $o$, such that $\{U_i \} \cup \{ N_j\}$ at $o$ form an orthornormal basis of $T_o \G/\Hh$. Then, the mean curvature vector of $\G_2\cdot o$ at $o$ is given by
\[
	\sum_{i,j} g(\nabla_{U_i} N_j, U_i)_o \,  N_j   = - \sum_{i,j} g(N_j, \nabla_{U_i} U_i)_o N_j.
\]
Since the last expresion is tensorial in $N_j$, by homogeneity we may replace the $N_j$ by Killing fields $X_j\in \ggo$  with  $X_j(o)=N_j(o)$. By \eqref{eqn_orth} and $\rg \subset \ggo_2$ we have that $X_j\in \lgo$. Moreover by adding elements in the isotropy $\hg \subset \lgo$ we may assume that $X_j \in (\mg_1 \oplus \mg_2)\cap \lgo$. For each $j$, formula \cite[(7.27)]{Bss} for the Levi-Civita connection of Killing fields then yields 
\begin{align*}
- \sum_i g(X_j, \nabla_{U_i} U_i) 
		=  \sum_i g([X_j,U_i], U_i) + \sum_k g([X_j, Z_k], Z_k) =   \tr (\ad X_j) |_{\ggo_2}.
\end{align*}
Here,
we have extended $g$ from $\mg = \mg_1\oplus \mg_2$ to all of $\ggo$ by making $\hg \perp \mg$, and  we denoted by $\{Z_k\}_k$ an orthonormal basis of $\hg_2\subset \ggo_2$. Then, the second
equality is justified by the fact that $[X_j, \hg_2] \subset \mg_2 \perp_g \hg_2$ using that $[\mg_1,\hg_2] = 0$.  

Observe now that since $\ggo_2$ is an ideal, the map $\lgo\ni X \mapsto (\ad X)|_{\ggo_2}$ is a Lie algebra representation.
Using that $\lgo$ is semisimple we may write $X_j = [Y_j, Z_j]$,  which 
implies  $\tr (\ad X_j) |_{\ggo_2} = \tr [(\ad Y_j) |_{\ggo_2}, (\ad Z_j) |_{\ggo_2}] = 0$  and the proof is finished.
\end{proof}

Let $\bar \Ll_1 \leq \Ll_1$ be the image of 
an Iwasawa subgroup $\bar \G$ of $\G$ 
under the quotient  $\pi_2 : \G \to \G_2 \backslash \G$.

\begin{lemma}\label{lem_barL1}
The group $\bar \Ll_1$ acts on $(\Ll_1/\Hh_1, \bar g)$ with cohomogeneity one, embedded orbits and orbit space $S^1$. 
\end{lemma}

\begin{proof}
We know by Proposition \ref{prop_periodicity} that the action of $\Go$ on $\G/\Hh$ has those properties. Therefore, to prove the lemma it suffices to show that $\pi$ gives a bijection between $\bar \G$-orbits in $\G/\Hh$ and $\bar \Ll_1$ orbits in $\Ll_1/\Hh_1$. To see this, consider the commutative diagram
\[
\begin{tikzcd}
		\G  \arrow[r, "\pi_2"]  \arrow[d, "\pi_\Hh"] & \G_2 \backslash \G=\Ll_1 \arrow[d, "\pi_1"] \\
		\G/\Hh  \arrow[r, "\pi"] &  \G_2 \backslash \G / \Hh =\Ll_1/\Hh_1
\end{tikzcd}
\]
The map $\pi_2$ is a Lie group morphism, and $\pi_1$ is $\Ll_1$-equivariant ($\Ll_1 = \G_2\backslash \G$). From this one easily deduces that $\pi$ maps $\bar \G$-orbits onto $\bar \Ll_1$-orbits: $\pi(\bar \G \cdot p) = \bar \Ll_1 \cdot \pi(p)$. Now assume that for some $p,q \in \G/\Hh$, $\pi(p)$ and $\pi(q)$ belong to the same $\bar \Ll_1$-orbit, say $\pi(p) = \bar l_1 \cdot \pi(q)$ for some $\bar l_1 = \pi_2(\bar x) \in \bar\Ll_1$, $\bar x\in \Go$, and let us show that $\Go \cdot p = \Go \cdot q$. Write $p = x\Hh$, $q=y\Hh$, $x,y\in \G$. By the commutative diagram and $\Ll_1$-equivariance of $\pi_1$ we have
\[
	\pi(x\Hh) = \bar l_1 \cdot \pi(y\Hh) = \pi_2(\bar x) \cdot \pi_1 (\pi_2(y)) = \pi_1( \pi_2 (\bar x y) ) = \pi(\bar x y \Hh).
\]
This implies that $p$ and $\bar x \cdot q$ belong to the same $\G_2$-orbit. Since $\G_2 \leq \Go$ and $\bar x\in \Go$, it follows that $\Go \cdot p = \Go \cdot q$.
\end{proof}

If $N$ is a unit normal field to the $\bar \G$-orbits in $\G/\Hh$, then $N$ is basic for the Riemannian submersion $\pi$ and the corresponding vector field $\bar N$ on $\Ll_1/\Hh_1$ is a unit normal field to the $\bar\Ll_1$-orbits. Also, a normal geodesic $\gamma(t)$ to the $\bar \G$-orbits, $\gamma(0) = e\Hh$, is horizontal for $\pi$, thus 
$\bar \gamma(t)=\pi\circ \gamma (t)$ is a normal geodesic to the $\bar \Ll_1$-orbits. The next lemma finally reduces the proof of Theorem \ref{thm_minimal} to the case where $\G$ is simple:

\begin{lemma}\label{lem_redss}
The mean curvatures of $\bar \G \cdot \gamma(t)$ and $\bar \Ll_1 \cdot \bar \gamma(t)$ coincide for all $t\in \RR$. 
\end{lemma}

\begin{proof}
An orthonormal basis for $T_{\gamma(t)} \big(\bar \G \cdot \gamma(t) \big)$ might be chosen by evaluating the vector fields $\{ X_i\} \cup \{ U_j\}$ at $\gamma(t)$, with $X_i$ basic for $\pi$ (horizontal and projectable) and $U_j\in \ggo_2$ Killing fields. The mean curvature of $\bar \G \cdot \gamma(t)$ at $\gamma(t)$ is thus given by
\[
	\sum_i g(\nabla_{X_i} N, X_i)_{\gamma(t)} + \sum_j g(\nabla_{U_j} N, U_j)_{\gamma(t)}.
\]
The second sum vanishes thanks to Lemma \ref{lem_Rmin}. Regarding the first one, notice that being basic, the $X_i$ are $\pi$-related to vector fields 
$\bar X_i$ in $\Ll_1/\Hh_1$, which span the tangent space to $\bar \Ll_1 \cdot \bar \gamma(t)$ at $\bar \gamma(t)$. Hence, 
$	\sum_i g(\nabla_{X_i} N, X_i)_{\gamma(t)} = \sum_i \bar g(\bar \nabla_{\bar X_i} \bar N, \bar X_i)_{\bar \gamma(t)}$ 
by a well-known property of Riemannian submersions,
which is by definition the mean curvature of $\bar \Ll_1 \cdot \bar \gamma(t)$ at $\bar \gamma(t)$.
\end{proof}

Before finishing the proof of Theorem \ref{thm_minimal}, we briefly recall the definition of an Iwasawa decomposition: see \cite{Knapp2e} or \cite{Helgason01}. Let $\lgo$ be the Lie algebra of a simple Lie group $\Ll$ and consider a Cartan decomposition $\lgo = \kg \oplus \pg$. Let $\theta \in \Aut(\ggo)$ be the corresponding Cartan involution
with $\theta|_\kg = \Id_\kg$, $\theta|_\pg = -\Id_\pg$ and $\kf_\theta(X,Y) := -\kf(X, \theta Y)$ a positive definite inner product on $\lgo$, where  $\kf=\kf_\lgo$ denotes the Killing form of $\lgo$. Choose a maximal abelian subalgebra $\ag \subset \pg$, and consider the corresponding root space decomposition 
\[
	\lgo = \lgo_0 \oplus \bigoplus_{\lambda \in \Lambda} \lgo_\lambda\,,
	\quad
	  \lgo_\lambda = \{ X\in \lgo : [H, X] = \lambda(H) \cdot X, \,\,  \forall H \in \ag\} 
\]
$\lambda \in \ag^*$ and $\Lambda \subset \ag^*$ being the set of those nonzero $\lambda$'s for which $\lgo_\lambda \neq 0$. 
Recall that for $\lambda \in \Lambda$ we have $\theta \lgo_\lambda = \lgo_{-\lambda}$, from which $\Lambda = - \Lambda$ follows. Choose a notion of \emph{positivity} in $\ag^*$, so that the set of positive elements $\ag^*_+$ is invariant under addition, multiplication by positive scalars, and $\ag^*\setminus \{ 0\} = \ag^*_+ \dot\cup -\ag^*_+$. Set $\Lambda^+ := \Lambda \cap \ag^*_+$, and consider the nilpotent subalgebra 
\[
	\ngo := \bigoplus_{\lambda \in \Lambda^+} \lgo_\lambda.
\]
The subalgebra $\bar \lgo := \ag \oplus \ngo$ is solvable, and the corresponding connected Lie subgroup $\bar \Ll$ acts on the homogeneous space $\Ll / \Hh$ with cohomogeneity one and with all orbits principal. Here we are using that $\Ll/\K$ is an irreducible non-compact Hermitian symmetric space, thus $\Hh = [\K,\K]$.  For each $\lambda \in \Lambda^+$ we decompose 
$\lgo_\lambda \oplus \lgo_{-\lambda} =  \kg_\lambda \oplus \pg_\lambda$, $ \kg_\lambda \subset \kg$ and 
$\pg_\lambda \subset \pg$,
where $\kg_\lambda = \{ X + \theta X : X\in \lgo_\lambda \}$ and 
$\pg_\lambda = \{X - \theta X : X\in \lgo_\lambda \}$. Set 
\[
	\pg^+ := \bigoplus_{\lambda \in \Lambda^+} \pg_\lambda,
\]
so that $\pg = \ag \oplus \pg^+$ is a $\kf$-orthogonal decomposition. Finally, notice that 
\begin{equation}\label{eqn_knkp+}
	\kg \oplus \ngo = \kg \oplus \pg^+.
\end{equation}

\begin{proof}[Proof of Theorem \ref{thm_minimal}]
By Lemmas \ref{lem_barL1} and \ref{lem_redss}, we may assume that $\G/\Hh = \Ll/\Hh$ with $\Ll = \Ll_1$ simple and with finite center. Let $\Ll = \K \Aa \N$ be an Iwasawa decomposition and $\bar \Ll = \Aa \N \leq \Ll$  be the corresponding Iwasawa subgroup. Notice also,
that the mean curvature of $\bar \Ll$-orbits is  a function on the orbit space $S^1$. 

Firstly, we assume  that $\lgo \not \simeq \sog(2,m)$, $m\geq 3$.
We will show in this case, that all $\bar \Ll$-orbits are minimal hypersurfaces. To that end, let $p = x\Hh \in \Ll/\Hh$. Using that $x\in \Ll = \bar \Ll Z(\K)\Hh$, we may assume without loss of generality that $p = k \Hh$ with $k \in Z(\K)$. Then at $p$  the isotropy subgroup is also $\Hh$, and a reductive decomposition  is given by $\lgo = \hg \oplus \mg$, $\mg = \zg(\kg) \oplus \pg \simeq T_p \Ll/\Hh$. 

Moreover, if in addition $\lgo\not \simeq \slg(2,\RR)$ then $\zg(\kg)$ and $\pg$ are the two inequivalent irreducible summands for the isotropy representation: see Remark \ref{rem_euclhom}. In particular, the metric $g$ restricted to $\pg$ is a multiple of the Killing form. Notice also that for any $Z\in \zg(\kg)$ the linear map $\ad Z$ preserves $\pg$, and $(\ad Z)|_\pg$ is skew-symmetric with respect to  $g$. 

The mean curvature of $\bar \Ll \cdot p$ at $p$ is given by
$\tr L_p = \sum_{i=1}^{n-1} g(\nabla_{X_i} X_i , N)_p$,
where $N$ denotes the unit normal to $\bar \Ll\cdot p$ and $\{ X_i\}_{i=1}^{n-1}$ is a $g_p$-orthonormal basis for $\bar \lgo$. The subspace $\zg(\kg) \oplus \pg^+ \subset \mg$ is the $g$-orthogonal complement of $\ag$ in $\mg$. By \eqref{eqn_knkp+} we have
$
	(\zg(\kg) \oplus \ngo) \cdot p = (\zg(\kg) \oplus \pg^+) \cdot p
	\subset T_p M$.
Since $N(p) \in  (\bar \lgo\cdot p)^\perp \subset (\ag \cdot p)^\perp  =  (\zg(\kg) \oplus \pg^+) \cdot p$, we may choose Killing fields $Z\in 
\zg(\kg)$, $X_\ngo\in \ngo$, such that $(Z+X_\ngo)(p) = N(p)$. Using the formula   \cite{Bss}[(7.27)]
for the Levi-Civita connection on Killing fields, we compute:
\begin{eqnarray}
	- \tr L_p  = \sum_{i=1}^{n-1} g([Z, X_i], X_i)_p +  \sum_{i=1}^{n-1} g([X_\ngo, X_i], X_i)_p\,.\label{eqn_traceL}
\end{eqnarray}
The first term vanishes. Indeed, write $X_i = Z_i + P_i$ where $Z_i \in \kg$, $P_i \in \pg$. Then $[Z, X_i] = [Z, P_i] \in \pg$, and using that $\pg \perp \zg(\kg)$, each summand equals $\la[Z, P_i], P_i \ra $, which vanishes because $(\ad Z)|_\pg$ is skew-symmetric. The second term equals $\tr (\ad_{\bar\lgo} X_\ngo) $, and this vanishes because $\ad_{\bar \lgo} X_\ngo$ is a nilpotent endomorphism.

If on the other hand we had $\lgo \simeq \slg(2,\RR)$ then $\hg = 0$. Choose a Killing field $X^\perp \in \slg(2,\RR)$ so that $X^\perp(p) = N(p)$ and notice that since  $\slg(2,\RR)$ is unimodular,
\[
	-\tr L_p  =  -\sum_{i=1}^{2} g(\nabla_{X_i} X_i , X^\perp)_p =  \sum_{i=1}^{2} g([X^\perp, X_i], X_i)_p = \tr \ad_\lgo X^\perp = 0.
\]
Finally, let us deal with the most complicated 
case $\lgo \simeq \sog(2,m)$, $m\geq 3$.
To see that we may assume that $\Ll$ is centerless,
denote by $\Zz \leq \Ll$ its center, a finite normal subgroup contained in $\K$ by \cite{HlgNeb}[Prop. 16.1.7].
Since $\Zz$ normalizes $\bar \Ll$, its action on $\Ll/\Hh$ by left-multiplication sends $\bar \Ll$-orbits to $\bar \Ll$-orbits. Thus, the mean curvature $\tr L_t$ considered as a function on the orbit space $\K/\Hh \simeq S^1$ is invariant under the action of $\Zz$ on said space. If we would know that on $\Zz \backslash \K / \Hh \simeq S^1 / \Zz \simeq S^1$ the mean curvature has mean zero, the same would follow for the mean curvature function on $\K/\Hh$. Then we may work with $\Zz \backslash \Ll$ instead.

We claim that in order to prove \eqref{eqn_intmin} it is enough to find an isometric involution $f: \Ll/\Hh \to \Ll/\Hh$ fixing $o$ with $ df_o N=-N$, which sends $\bar \Ll$-orbits to $\bar \Ll$-orbits.
Indeed, in the notation of Section \ref{sec_cohom1}, such an isometry must preserve $\Sigma_0 = \bar \Ll \cdot o$, $o = e\Hh = \gamma(0)$,
thus $ f(\gamma(t)) = \gamma(-t)$,  $f(\Sigma_t) = \Sigma_{-t}$
and $ df_{\gamma(t)} N_t = -N_{-t}$.
Since $f$ is an isometry of 
 $\Ll/\Hh$, the mean curvatures of $\Sigma_t$ and $\Sigma_{-t}$ with respect to $N_t$ and $-N_{-t}$ respectively, coincide. Since $\tr L_s$ is the mean curvature with respect to $N_s$, $\tr L_s$ is an odd function on $S^1$ and \eqref{eqn_intmin} holds.

 Let $\Ll$ be the connected, centerless simple Lie group with Lie algebra $\sog(2,m)$. Another group with two connected components and
 the same Lie algebra  is given by
 \[
 	\SO(2,m) := \{ A\in \Gl({2+m},\RR) : \la A v, A w\ra_{2,m} = \la v,w\ra_{2,m}, \quad  \forall v,w\in \RR^{2+m}, \quad \det A = 1 \}
 \]
where $\la x,x \ra_{2,m} = x_1^2 + x_2^2 - x_3^2 - \cdots - x_m^2$.
Thus
 \[
 \sog(2,m)=\left\{	 \left( \begin{array}{cc}
 						A 	& 	 B		\\
 						B^t	& 	C
 	\end{array} \right) \in \glg_{2+m}(\RR), \,\,\,A\in \sog(2), 
 	\,\,\, C\in \sog(m), \,\,\, B\in \RR^{2\times m}\right\}
 \]
Fix the maximal compact subgroup $\hat \K = S(\Or(2)\Or(m)) \leq\SO(2,m)$ containing the center $\Z$ of $\SO(2,m)$ by
\cite{HlgNeb}[Prop. 16.1.7].
Since $\Ll = \SO(2,m)/\Z$, 
 $\K := \hat \K /\Z$ is a maximal compact subgroup of $\Ll$
 with Cartan decomposition $\lgo = \kg \oplus \pg$. In the above matrix representation, $\kg$ is characterized by  $B=0$ and $\pg$ by  $A=C=0$. 

Let us now choose a maximal abelian subalgebra $\ag$ of $\pg$ as follows: 
 \[
 	\ag := \left\{  \left( \begin{array}{cc}
 						0	&  B \\
 						B^t	& 	 0	
 	\end{array} \right) \in \pg : B = 
 	\left( \begin{array}{cccc}
 						a_1	&  0 & \cdots &  0\\
 						0	& a_2 & \cdots & 0
 	\end{array} \right), \quad a_i\in \RR
 	 \right\}.
 \]
 Let $\lgo = \oplus_{\lambda\in \Lambda}$ $\lgo_\lambda$ be the root space decomposition according to $\ag$, and set $\ngo = \oplus_{\lambda \in \Lambda^+} \lgo_\lambda$, for some choice of positive roots $\Lambda^+ \subset \Lambda$. Then, the Iwasawa decomposition 
 \cite{Helgason01}[Ch. IX, Thm. 1.3] asserts 
 a global decomposition  $\Ll = \K \Aa \N$, 
that is $\Ll$ and $\K \times \Aa \times \N$ are diffeomorphic,
 where  $\Aa, \N$ denote the connected subgroups of $\Ll$ with Lie algebras $\ag, \ngo$, respectively. We set $\bar \Ll := \Aa \N \leq \Ll$, a closed, solvable subgroup acting on $\Ll / \Hh$ with cohomogeneity one and with orbits space $\SO(2) \simeq S^1$. 
 Notice that  $\Hh = [\K,\K] $ with
  $\hg \simeq \sog(m)$.

 The isotropy representation of $\Ll/\Hh$ has a one-dimensional trivial factor $\zg(\kg) = \sog(2)$, and $\pg$ is a sum of two irreducible and equivalent $\Ad (\Hh)$-modules of real type, since $m\geq 3$.
 Thus, given $g$, after rotating the first two coordinates in $\RR^{2\times m}$ we may assume without loss of generality that the two rows of $B$ in the matrix representation are $g$-orthogonal.

 Consider now the diagonal element of $x=\SO(2,m)$ whose non-zero entries are given by $x= {\rm diag}(-1,1,-1,1,\Id_{\RR^{m-2}})$
(and notice that $x \not \in \SO(2,m)_0$).
Since  conjugating by $x$ preserves $\Z$,
 it induces an automorphism $\hat f$ of $\Ll$ of order two. Also, $x$ normalizes $\hat \K$, hence $\hat f(\K) \subset \K$ and therefore $\hat f(\Hh) = \Hh$. In this way, $\hat f$ induces a diffeomorphism $f$ of $\Ll/\Hh$ with $f \circ f={\rm id}_{\Ll/\Hh}$ 
 fixing $o:= e\Hh$, via $f(l\Hh) := \hat f(l) \Hh$.
 
 The Lie algebra automorphism $\varphi:= d \hat f|_e : \lgo \to \lgo$, which one can identify also with conjugation by $x$ in the matrix representation of $\sog(2,m)$, \emph{fixes} the subalgebra $\ag$, 
 and therefore preserves the root spaces and in particular $\ngo$. Hence $\hat f(\bar \Ll) = \bar \Ll$ and $f (\bar \Ll \cdot p) = \bar \Ll \cdot f(p)$ for all $p\in \Ll/\Hh$.  Next, notice that for all $l \in \Ll$ we have
 $L_l \circ f =f \circ L_{f^{-1}(l)}$, $L_l(\tilde l \Hh)=l\tilde l\Hh$.
  Since $(df)_o:(\mg,g) \to (\mg,g)$,
 $\mg =\zg(\kg)\oplus \pg$, is an isometry, we see that
  $f^* g = g$. 

If remains to show that $f$ reverses the orientation of the normal geodesic. To that end notice that the $\SO(2)$-orbit through $o$ (also a geodesic) is also preserved by $f$, since $\hat f(\SO(2)) = \SO(2)$. But on Lie algebra level the action of $d \hat f|_e$ on $\sog(2)$ is simply $-\Id_{\sog(2)}$.
 Thus the orientation of this curve transversal to the $\Ss$-orbits is reversed. Clearly the same must hold for the normal geodesic, and this completes the proof. 
 \end{proof}

\section{Proof of Theorem \ref{thm_main}}\label{sec_proofmain}

Let $(\RR^n,g)$ be a homogeneous Einstein space with 
Einstein constant $-1$. Given any effective presentation $\G'/\Hh'=\RR^n$, by \cite[Thm.~4.6]{alek} and \cite[Thm.~0.2]{JblPet14} one can find an `improved' effective presentation $\G/\Hh$ with $\dim \G \leq \dim \G'$ satisfying $\hg \subset \lgo$, for some  Levi decomposition $\ggo = \lgo \ltimes \rg$. Indeed, notice that $\dim \ag \leq \dim \zg(\ug)$ in the notation of \cite[$\S$3.2]{JblPet14},  thus the dimension of the modified transitive group constructed in \cite[Prop.3.14]{JblPet14} cannot be larger.  In addition with respect to this presentation the metric satisfies \eqref{eqn_orth} by \cite[Thm.~2.4]{AL16}.

The above argument ensures that we may take $\dim \G$ to be minimal among all transitive groups, and still ensure that \eqref{eqn_orth} is satisfied. By doing so, we can apply Proposition \ref{prop_periodicity} and after quotienting by a discrete group of isometries we obtain another Einstein homogeneous space with Einstein constant $-1$, which for simplicity we denote by $(\G/\Hh,g)$, on which all Iwasawa subgroups $\bar \G \leq \G$ act with cohomogeneity one, embedded orbits and orbit space $S^1$. In short, $(\G/\Hh,g)$ has a Levi presentation (see Definition \ref{def_Levi}).

We turn to the main geometric result in this section: based on Theorem \ref{thm_main2}, we show that the simple factors are pairwise orthogonal.

\begin{proposition}\label{prop_simpleperp}
Suppose that $(\G/\Hh,g)$ has a Levi presentation and $\ricci_g = -g$.
Then, any two simple factors $\lgo_i,\lgo_j \subset \lgo$, $i\neq j$,  satisfy 
$\lgo_i \cdot o \perp_g \lgo_j \cdot o$, $o = e\Hh$.
\end{proposition}
 
\begin{proof}
Let $i=1$ and  $\lgo_{\geq 2}$ be the unique complementary ideal
of $\lgo_1$ in $\lgo$, so that $\lgo = \lgo_1 \oplus \lgo_{\geq 2}$. By Proposition \ref{prop_Rnhomognew}, the isotropy subalgebra decomposes as a direct sum of ideals $\hg = \hg_1 \oplus \hg_{\geq2}$ with $(\lgo_1, \hg_1)$ the infinitesimal data of an $\RR$-bundle over an irreducible Hermitian symmetric space $\Ll_1/\K_1$ of non-compact type, $\hg_1=\kg_1^{ss}$ and $\hg_{\geq 2} \subset \lgo_{\geq 2}$. 
By Lemma \ref{lem_maxab} there exists a maximal compactly embedded subalgebra of $\kg_1 \subset \lgo_1$ containing $\hg_1$ with $\hg_1 = [\kg_1,\kg_1]$. 

Let $\Go \leq \G$ be an Iwasawa subgroup corresponding
to $(\lgo_1, \kg_1)$ as given by Theorem \ref{thm_minimal}.
Since these $\Go$-orbits are integrally minimal,  by Theorem \ref{thm_main2} they are standard homogeneous spaces. In order to exploit this property, 
let $\lgo_1 = \kg_1 \oplus \ag_1\oplus \ngo_1$ be the Iwasawa decomposition for which $\sg_1 := \ag_1 \oplus \ngo_1 \subset \bar\ggo$,
and let $\ngo$ denote the nilradical of $\ggo$.
Then $\bar\ngo := \ngo_1 \ltimes \ngo$ is the nilradical of $\bar \ggo$. Set 
 $\bar \ug := \ngoo^\perp \subset \ggoo$,
$\ug:= \ngo^\perp \subset \ggo$,  $\ag := \ngo^\perp \subset \rg$, and recall that in order to define these orthogonal complements we extend the metric $g$ to all of $\ggo$  as in Definition \ref{def_stand}. Thus,
\[
    \ggo= \big( \rlap{$\overbrace{\phantom{\hg_1 \oplus  \zg(\kg_1) \oplus \ag_1\oplus\ngo_1}}^{\lgo_1}$}  \hg_1 \oplus  \zg(\kg_1) \oplus \underbrace{\ag_1\oplus\ngo_1 \oplus \lgo_{\geq 2}  \big) \ltimes  ( \rlap{$\overbrace{\phantom{\ag\oplus \ngo}}^{\rg}$}  \ag \oplus \ngo)}_{\bar \ggo}, \qquad \bar \ggo = \lgo_{\geq 2} \ltimes \big(\ag_1 \oplus \ag \oplus \rlap{$\overbrace{\phantom{\ngo_1\oplus \ngo}}^{\bar \ngo} $} \ngo_1 \oplus \ngo \big),
\]
 where $\oplus$ denotes vector spaces direct sum. We now claim that 
\begin{eqnarray}\label{eqn_ngperpl}
  \ngo_1 \cdot o \perp \lgo_{\geq 2} \cdot o\,.
\end{eqnarray}
To see that, notice that the subspaces $\ug$, $\bar \ug$ satisfy
$\bar \ug \subset \ug \cap \bar \ggo = \sg_1\oplus \lgo_{\geq 2}\oplus \ag$,  the last equality thanks to \eqref{eqn_orth}.
Also, since $(\G/\Hh, g)$ 
is Einstein, it is also standard by \cite{alek}. Thus,
the orthogonal complements $\ug$ and $\bar \ug$ are both  Lie subalgebras
of $\ggo$. Since $\bar \ug \simeq \bar \ggo/\bar \ngo$ is reductive (that is, the direct sum of a semisimple ideal and the center), $[\bar \ug, \bar \ug]$ is semisimple and isomorphic to a Levi factor $\lgo_{\geq 2}$ of $\bar \ggo$ (see \cite[3.16.3--4]{Varad84}). Moreover,
$	[\bar \ug, \bar \ug] \subset [\sg_1, \sg_1] \oplus \lgo_{\geq 2}$
using that $[\ggo, \rg] \subset \ngo$: see \cite[ 3.8.3]{Varad84}.
Since $ [\sg_1, \sg_1] \oplus \lgo_{\geq 2}$  is a Lie algebra direct sum of a solvable and a semisimple ideal, it has exactly one semisimple subalgebra isomorphic to $\lgo_{\geq 2}$. Hence $[\bar \ug, \bar \ug] = \lgo_{\geq 2}$ and in particular, $\ngo_1 \subset \bar \ngo \perp \bar \ug \supset \lgo_{\geq 2}$, which shows \eqref{eqn_ngperpl}.

We consider now two different cases. Assume first that $\hg_1\neq 0$,
that is, $\lgo_1 \not\simeq \slg(2,\RR)$. 
Then, the reductive complement $\pg_1$ of the symmetric pair $(\lgo_1,\kg_1)$ is an $\Ad(\Hh)$-isotypical summand in $\Ll/\Hh$: see Remark \ref{rem_euclhom}. Thus, $\pg_1 \cdot o \perp \lgo_{\geq 2}\cdot o$. 
On the other hand,  $\ngo_1 \cdot o \perp \lgo_{\geq 2} \cdot o$ 
by \eqref{eqn_ngperpl}. 
If we would know that $\ngo_1 \cdot o \not\subset \pg_1 \cdot o$, 
then from the fact that $\pg_1 \cdot o$ is of codimension one inside $\lgo_1\cdot o$, we could conclude that $\lgo_1 \cdot o \perp \lgo_{\geq 2} \cdot o$ and the proof would follow in this case.
 Suppose on the contrary that $\ngo_1 \cdot o \subset \pg_1 \cdot o$, which is equivalent to $\ngo_1 \oplus \hg_1\subset \pg_1 \oplus \hg_1$. Since $\ag_1 \subset \pg_1$,  this yields
$\sg_1 \oplus \hg_1\subset \pg_1 \oplus \hg_1$ and by
counting dimensions one gets
$\sg_1 \oplus \hg_1= \pg_1 \oplus \hg_1$. 
Using this several times we observe that
\[
	[\pg_1, \pg_1] \subset [\hg_1 \oplus \sg_1, \hg_1 \oplus \sg_1] \subset \hg_1 + [\sg_1,\sg_1] + [\hg_1,\sg_1] \subset \hg_1 + \sg_1 + [\hg_1, \hg_1\oplus \pg_1] \subset \hg_1 \oplus \pg_1,
\]
contradicting the fact that $[\pg_1,\pg_1]=\kg_1$ (see \cite[Prop. 13.1.10]{HlgNeb}). 

Finally suppose that 
$\lgo_1 \simeq \slg(2,\RR)$. Again we have $\ngo_1 \cdot o \perp \lgo_{\geq 2} \cdot o$ by \eqref{eqn_ngperpl}. We may apply this argument to different Iwasawa groups $\bar \G$ defined by considering other Iwasawa decompositions for $\slg(2,\RR)$, which can be obtained by conjugating with elements in $\SO(2)$. By doing so with three different decompositions whose nilradicals are linearly independent, we conclude that $ \lgo_1 \cdot o \perp \lgo_{\geq 2} \cdot o$, and this finishes the proof.
\end{proof}

Recall the following result, which is a particular case of the main theorem in \cite{Nkn2}:

\begin{theorem}\cite{Nkn2}\label{thm_niko}
Let $(\Ll/\Hh, g)$ be a homogeneous space with $\Ll$ semisimple and $\Hh$ compact, and consider a Cartan decomposition $\lgo = \kg \oplus \pg$ with $\hg \subset \kg$ and a reductive complement $T_o \Ll/\Hh \simeq \mg = \qg \oplus \pg$ with  $\hg \oplus \qg = \kg$. Assume that $g$ is \emph{awesome}, that is, $\kg\cdot o \perp_g \pg\cdot o$ for $o=e\Hh$. If in addition  $\Ricci_g |_{\pg\times\pg} = c \cdot g|_{\pg\times \pg}$ for some $c\in \RR$, then either $\hg = \kg$, or there exists $Z\in \kg$ with $\Ricci_g(Z,Z) > 0$. 
\end{theorem}

Using this, we are now in a position to prove the following result, which implies Theorem \ref{thm_main}:

\begin{theorem}\label{thm_Gsolv}
If $(\G/\Hh,g)$ has a Levi presentation and $\ricci_g=-g$, then $\G$ is solvable.
\end{theorem}

\begin{proof}
 By Proposition \ref{prop_simpleperp} and Remark \ref{rem_euclhom}, the induced metric on $\Ll/\Hh$ is awesome. In the semisimple case $\G = \Ll$,  awesomeness  
implies the non-existence of Einstein metrics by Theorem \ref{thm_niko}. 
To deal with the general case we assume that $\lgo\neq 0$ and use 
the structure results \cite[Theorems 2.1 $\&$ 2.4]{AL16}
for non-compact  homogeneous Einstein spaces. These results yield that at $o:= e\Hh$, the induced homogeneous metric $g_{\Ll/\Hh}$ on $\Ll\cdot o \simeq \Ll/\Hh$ has
\[
	\Ricci_{g_{\Ll/\Hh}} = - g_{\Ll/\Hh} + C_\theta, \qquad C_\theta(X,Y) = \unc \tr \big(\theta(X) + \theta(X)^t \big) \big(\theta(Y) + \theta(Y)^t\big),
\]
where $\theta : \lgo \to \End(\rg)$ is the restriction of the adjoint representation: $\theta(X) = (\ad X)|_{\rg}$, and as usual $T_o \Ll/\Hh \simeq \mg$ for some reductive decomposition $\lgo = \hg \oplus \mg$. Now let $\Ll_1 \leq \Ll$ be the connected Lie subgroup with Lie algebra $\lgo_1$, one of the simple ideals in $\lgo$, and set $\Hh_1 := \Ll_1 \cap \Hh$. It follows by Proposition \ref{prop_simpleperp} that $(\Ll/\Hh, g_{\Ll/\Hh})$ is, at least locally, a Riemannian product of the homogeneous spaces given by the orbits of each of the simple factors. Thus, the induced homogeneous metric $g_1$ on $ \Ll_1 \cdot o \simeq \Ll_1/\Hh_1$ has Ricci curvature also given by
\begin{equation}\label{eqn_Ricg1}
	\Ricci_{g_1} = -g_1 + C_{\theta_1}, \qquad \theta_1 := \theta|_{\lgo_1} : \lgo_1 \to \End(\rg).
\end{equation}
By a result of Lauret (Proposition A.1 in the Appendix to \cite{semialglow}), awesomeness implies  
$\theta(X)^t = -\theta(X)$ for all $ X\in \kg$
and  $\theta(X)^t = \theta(X)$ for all $X\in \pg$,
where $\lgo = \kg \oplus \pg$ is the Cartan decomposition with $\kg \cdot o \perp \pg \cdot o$. In particular
$C_\theta|_{\kg \times \kg} = C_\theta|_{\kg \times \pg} = 0$
and
\[
	 \qquad C_\theta(X,Y) = \tr \big(\theta(X)\theta(Y) \big)
\]
for all $X,Y\in \pg$. Since  $C_\theta$ 
 is an $\Ad(\Ll)$-invariant bilinear form on $\lgo$, 
 its restriction to each simple factor is a multiple of the Killing form of that factor. Together with \eqref{eqn_Ricg1}, this yields 
\[
	\Ricci_{g_1}|_{\kg_1\times \kg_1}  = -g_1, \qquad \Ricci_{g_1}|_{\pg_1\times \pg_1} = a \cdot g_1, \qquad a\in \RR.
\]
The second identity allows us to apply Theorem \ref{thm_niko} to the homogeneous space $(\Ll_1/\Hh_1, g_1)$, and this contradicts the first identity (recall that $\kg_1$ is not contained in the isotropy, by the results of Section \ref{sec_homspaces}). Therefore, $\ggo = \rg$ as claimed.
\end{proof}

\begin{proof}[Proof of Theorem \ref{thm_main}]
By the observations made at the beginning of this section, given a homogeneous space $(\RR^n,g)$ with $\ricci_g = -g$, after passing to a quotient we may assume that it has a Levi presentation. Then by Theorem \ref{thm_Gsolv}, $(\RR^n,g)$ is a solvmanifold.
\end{proof}

Another immediate consequence of Theorem \ref{thm_Gsolv} is the following

\begin{corollary}\label{cor_nonexist}
The group $\Sl(2,\RR)^k$  admits no left-invariant Einstein metrics for any $k\geq 1$.
\end{corollary}

\begin{proof}
Since the group is semisimple and there is no isotropy, it is clear that the presentation is a Levi presentation. Moreover, by \cite{Mln} it is well-known that this group does not admit flat metrics.  Therefore, by non-compactness any left-invariant Einstein metric would satisfy $\ricci_g = -g$ up to scaling, contradicting Theorem \ref{thm_Gsolv}.
\end{proof}

Since the Einstein condition is local, the same statement is of course true for any Lie group locally isomorphic to $\Sl(2,\RR)^k$, such as its universal cover. It is interesting to remark that this result is in fact not a consequence of Theorem \ref{thm_main}, because there exist left-invariant metrics on $\Sl_*(2,\RR)$ which are isometric to a solvmanifold, see \cite{GJ15}. Also, let us point out that this also implies non-existence of left-invariant Ricci soliton metrics on these Lie groups, since by \cite[Thm.~1.3]{Jbl2015} they must be Einstein.

Finally, we prove another application of Theorem \ref{thm_main}, stated in the introduction:

\begin{proof}[Proof of Corollary \ref{cor_main1}]
Let $(\RR^n,g)$ be a homogeneous Ricci soliton with $\ricci_g + 
 \lca_X g= \lambda \cdot g $. For $\lambda \geq 0$, such solitons are 
 by \cite{Nab10}, \cite{PW} quotients of the Riemannian product of a compact homogeneous Einstein manifold and a flat Euclidean space. Since our underlying manifold is $\RR^n$,
  we may therefore assume $\lambda<0$: see \cite{Mln}. 
 By \cite{Jbl13b} there exists a presentation $(\G/\Hh,g)$ which makes $(\RR^n,g)$ an algebraic soliton: the Ricci endomorphism at $e\Hh$ is given by $\Ricci_g = \lambda \cdot \Id + {\rm D}$, with ${\rm D}$ the projection onto $\ggo/\hg \simeq T_{e\Hh} \G/\Hh$ of a derivation of $\ggo$. Notice that ${\rm D}$ is symmetric, and in particular a normal operator. Using Theorem 3.2 from \cite{HePtrWyl} we obtain a one-dimensional 
 homogeneous Einstein extension $(\RR^{n+1} = \hat \G / \Hh, \hat g)$
 with $\ricci_{\hat g} = \lambda \cdot \hat g$ (with  pairwise isometric
 $\G$-orbits, because  $\G$ is normal in $\hat \G$).
 
 Theorem \ref{thm_main} and Theorem 1.3 in \cite{Jab15} 
  imply now that 
 $(\RR^{n+1}, \hat g)$ is an Einstein solvmanifold, which
 is strongly solvable. As a consequence, by the very definition of strongly solvable, there exists a solvable subgroup $\hat \Ss$ of $\hat \G$ which acts transitively on $(\RR^{n+1}, \hat g)$. Applying 
 well-known properties of solvmanifolds 
 we may also assume that $\dim \hat \Ss = n+1$, see \cite{GrdWls}. 
 
 Let $\hat \ggo,\ggo,\hat\sg$ denote the Lie algebras of Killing fields
 of $(\RR^{n+1},\hat g)$ corresponding to the isometric group actions of $\hat \G, \G, \hat \Ss$, respectively, and set $\sg := \hat \sg \cap \ggo$. Since $\ggo$ has codimension one inside $\hat \ggo$, $\sg$ has codimension one inside $\hat \sg$. Moreover, $\sg$ spans the tangent space to the $\G$-orbits in $\hat \G/\Hh$, since otherwiser $\hat \Ss$ would not be transitive in $(\RR^{n+1}, \hat g)$. The corresponding connected Lie subgroup $\Ss$ of $\G$ acts transitively on $\G/\Hh$, thus $(\G/\Hh,g)$ is a solvmanifold and finally by \cite[Thm.~1.1]{Jbl2015} it is isometric to a solvsoliton, as claimed.  
\end{proof}

\begin{appendix}

\section{The space of invariant metrics}\label{app_homogeneous}

Let $\G/ \Hh$ be a homogeneous space. Fix a background $\G$-invariant Riemannian metric $\bar g$ and consider the canonical reductive decomposition $\ggo = \hg \oplus \mg$, where $\mg$ is the orthogonal complement of $\hg$ with respect to the Killing form of $\ggo$. The space $\mca^\G(\G/\Hh)$ of $\G$-invariant Riemannian metrics on $\G/ \Hh$ can be identified with the set of $\Ad(\Hh)-$invariant inner products on $\mg$. The latter is an orbit 
\[
	\GHm \cdot \bar g \subset \Sym^2(\mg),
\] 
where $\bar g$ denotes also the inner product induced  on $\mg \simeq T_{e\Hh} \G / \Hh$ by the background metric.
Here $\Sym_2(\mg)$ denotes the space of symmetric bilinear forms on $\mg$, $\Gl(\mg)$ acts on it by the usual change of basis action
\begin{equation}\label{eqn_Sym2action}
	(Q\cdot b)(\,\,\cdot \,,\,\cdot \,):=b(Q^{-1}\cdot \,,Q^{-1}\cdot \,), \qquad Q\in \Gm, \quad b\in \Sym^2(\mg).
\end{equation}
and $\GHm$ is the centralizer of $\Ad(\Hh)|_{\mg}$ in $\Gl(\mg)$. Thus, if $\Om \subset \Gm$ denotes the subgroup of $\bar g$-orthogonal endomorphisms, and $\OHm = \Om \cap \GHm$,  then the set of $\Ad(\Hh)-$invariant inner products on $\mg$ is also a homogeneous space
\[
	\mca^\G(\G/ \Hh) \simeq \GHm / \OHm.	
\]
Moreover, if $\Q \subset \GHm$ is a closed subgroup big enough so that $\GHm = \Q \OHm$ then it still acts transitively on $\GHm / \OHm$. In this way we also get a presentation
\[
	\mca^\G(\G/\Hh) \simeq \Q / \K_\Q, 
\]
where $\K_\Q =  \Q \cap \OHm$ is a maximal compact subgroup of $\Q$. 

\begin{lemma}\label{lem_lift}
Let $(g_t)_{t\in I} \subset \mca^\G(\G/\Hh) \simeq \Q / \K_\Q$ be a smooth family of metrics.
Then, there exists a smooth lift $(q_t)_{t\in I} \subset \Q$ 
such that $g_t = q_t \cdot \bar g$ for all $t\in I$.
\end{lemma}

The reductive decomposition induces an inclusion
\begin{equation}\label{eqn_Gminclusion}
	\Gm  \simeq \minimatrix{\Id_{\hg}}{}{}{\Gl(\mg)} \subset \Gg\,, 
\end{equation} 
according to which we set
\[
	\Aut^{\Hh}_{\mg}(\ggo) := \Aut(\ggo) \cap \GHm\, ,
\]
where $\Aut(\ggo)$ denotes the Lie group of automorphisms of the Lie algebra $\ggo$.


\begin{lemma}\label{lem_Adf}
Let $g$ be a $\G$-invariant metric on $\G/ \Hh$, set $o := e\Hh$ and let $f\in \G$ normalizing $\Hh$. Then, the scalar products $g_o$, $g_{f\cdot o}$ induced respectively on $\mg$  by $g$ after the identifications $\mg \simeq T_o \G/\Hh$, $\mg \simeq T_{f\cdot o} \G/ \Hh$
satisfy
\[
	g_{f \cdot o}(\, \cdot \, ,  \, \cdot \,) = g_o \big( {(\Ad f^{-1})|_\mg  \, \cdot \, , (\Ad f^{-1})|_\mg \cdot \, }\big),
\]
with $(\Ad f^{-1})|_\mg \in \Aut^{\Hh}_{\mg}(\ggo) $.
\end{lemma}

\begin{proof}
Since the map $f$ is an isometry, for a Killing field $X\in \mg$ with
flow $\Phi_s$ the vector field $\tilde X$ given by 
$\tilde X_{f\cdot p} := df_p X_p$ is again a Killing
field, since its flow satisfies $\tilde \Phi_s = f \circ \Phi_s \circ f^{-1}$.
As a consequence $\tilde X = (\Ad f)\vert_{\mg} X \in \mg$. 
Setting $p:= f\cdot o$, we deduce
\begin{equation}\label{eqn_Killc}
   (d f)_o X_o = \big((\Ad f) X \big)_{p}\,. 
\end{equation}
Using that $f$ normalizes $\Hh$ we see that the isotropy subgroup at $p$ is also $\Hh$. In particular, we also have $\mg \simeq T_p \G/ \Hh$ under the usual identification $X \mapsto X_p$.
Hence, for $X\in \mg$ we have
\begin{align*}
	 g_o (X_o, X_o)& =  g_{p}\big(  (df)_o X_o, (df)_o X_o \big)\\
		& = g_{p} \big( (\Ad f) X_{p}, (\Ad f) X_{p}  \big) 
		= g_p \big( (\Ad f)|_\mg X_p, (\Ad f)|_\mg X_p \big),
\end{align*}
and the lemma follows.
\end{proof}

An analogous of the above lemma holds of course for any $\G$-invariant tensor.

\begin{corollary}\label{cor_shape}
Let $(g(t))_{t \in \RR}$ be a smooth curve of homogeneous metrics on
$\G/\Hh$ and $L_t$ be defined by 
$g_t' ( \,\cdot \, ,\, \cdot )
	=2\, g_t(L_t \,\cdot \, ,\, \cdot )$.
Suppose that there exists $f \in \G$ normalizing $\Hh$
and $T>0$ such that $g_{T+t}= (\Ad f)|_{\mgo} \cdot g_t$ for all
$t \in \RR$. Then
 \begin{eqnarray*}
     L_T &=& (\Ad f)|_{\mgo}  \circ  L_0 \circ (\Ad f)|_{\mgo}^{-1}\,.
 \end{eqnarray*}
\end{corollary}

\begin{proof}
We write $g_t (\cdot \,,\,\cdot \,)= \bar g(G_t\,\cdot \,,\,\cdot \,)$,
where $\bar g$ is a $\Go$-invariant background metric on $\Sigma_0$.
Then by Lemma \ref{lem_Adf} we have for all $t \in \RR$ that
$G_{T+t}=((\Ad f)|_{\mgo}^{-1})^T G_t (\Ad f)|_{\mgo}^{-1}$, the transpose
taken with respect to $\bar g$. Using $G_t'=2G_tL_t$ 
implies the claim.
\end{proof}

\begin{lemma}\label{lem_Sfactor}
Semisimple subgroups of a semisimple Lie group are closed.
\end{lemma}

\begin{proof}
Let $\Ll$ be a Lie group whose Lie algebra splits as a direct sum of semisimple ideals $\lgo = \lgo_1 \oplus \lgo_2$. Let $\Z(\Ll)$ be the center of $\Ll$ and consider the universal cover $\tilde \Ll$ and the centerless quotient $\check \Ll := \Ll / \Z(\Ll)$. Let $\Ll_i$ (resp. $\tilde \Ll_i, \check \Ll_i$) be the connected Lie subgroup of $\Ll$ (resp. $\tilde \Ll, \check \Ll$) with Lie algebra $\lgo_i$, $i=1,2$. 

It is enough to show that $\Ll_1$ is a closed subgroup of $\Ll$. For the universal cover this is clear,  as $\tilde \Ll$ is isomorphic to the direct product  $\tilde \Ll_1 \times \tilde \Ll_2$. Notice that $\check \Ll \simeq \tilde \Ll / \Z(\tilde \Ll)$, thus $\check \Ll_i \simeq \tilde \Ll_i / \Z(\tilde \Ll_1)$ and $\check\Ll$ is  isomorphic to the direct product of the centerless subgroups $\check \Ll_1 \times \check \Ll_2$. Assume on the contrary that $\Ll_1$ is not closed in $\Ll$. This means that its closure $\bar \Ll_1$ must have strictly larger dimension, and hence its intersection with $\Ll_2$ must have positive dimension. Let $\pi : \Ll \to \check \Ll$ denote the projection onto the quotient. Then $\pi(\Ll_i)$ is connected and with Lie algebra $\lgo_i$, hence equal to $\check \Ll_i$. On the other hand,
\[
	\pi(\bar \Ll_1 \cap \Ll_2) \subset \pi(\bar \Ll_1) \cap \pi(\Ll_2) \subset \overline{\pi(\Ll_1)}\cap \pi(\Ll_2) = \check \Ll_1 \cap \check \Ll_2 = \{e\}.
\]
From this, $\bar \Ll_1 \cap \Ll_2 \subset \Z(\Ll)$, and this is a contradiction since $\Z(\Ll)$ is discrete.
\end{proof}

\section{The endomorphism $\Beta$ associated to a homogeneous space}\label{app_beta}

Let $\G/\Hh$ be a homogeneous space endowed with a fixed, background $\G$-invariant Riemannian metric $\bar g$, and assume that $\Hh$ is  compact. Let $\ggo = \hg \oplus \mg$ be the canonical reductive decomposition, and extend $\bar g$ from $\mg$ to an $\Ad \Hh$-invariant scalar product on all of $\ggo$, also denoted by $\bar g$, and such that $\bar g(\hg,\mg) = 0$. In this section, the notions of symmetric and orthogonal endomorphisms of $\ggo$ are considered with respect to $\bar g$ only.

Given \emph{any} symmetric endomorphism $\Beta \in \End(\ggo)$ we now introduce the associated subgroups $\Qb, \Slb$ of $\Gl(\ggo)$ following \cite{GIT}, and we refer the reader to that article and \cite{BL17} for proofs of the results in this section.
Consider the adjoint map 
$\ad(\Beta) : \End(\ggo) \to \End(\ggo)$, $A \mapsto [\Beta, A],$ which is symmetric with respect to the scalar product $\tr AB^t$ on $\End(\ggo)$.  If $\End(\ggo)_r$ denotes the eigenspace of $\ad(\Beta)$ with eigenvalue $r\in \RR$ then $\End(\ggo) = \bigoplus_{r\in \RR} \End(\ggo)_r$, and we set
\[
    \ggo_\Beta := \End(\ggo)_0 = \ker (\ad(\Beta) ), \qquad \ug_\Beta := \bigoplus_{r> 0} \End(\ggo)_r, \qquad \qg_\Beta := \ggo_\Beta \oplus \ug_\Beta.
\]

\begin{definition}\label{def_groupsapp}
We denote respectively by 
\[
    \Gb := \{ g \in \G : g \Beta g^{-1} = \Beta \}, \quad \Ub := \exp(\ug_\Beta)\quad \textrm{and}
     \quad \Qb := \Gb \Ub
\] 
the centralizer, the unipotent subgroup and the parabolic subgroup associated with $\Beta$. 
\end{definition}

Of course their Lie algebras are respectively $\ggo_\Beta$, $\ug_\Beta$ and $\qg_\Beta$. The group $\Ub$ is nilpotent, and $\Gb$ is reductive with Cartan decomposition $\Gb = \Kb \exp(\pg_\Beta)$, where $\Kb = \Or(\ggo) \cap \Gb$ and $\pg_\Beta = \{A\in \ggo_\Beta : A= A^t \}$. Notice that $\Beta \in \pg_\Beta$, and that $\hg_\Beta := \{A\in \pg_\Beta : \tr A \Beta = 0 \}$ is a codimension-one Lie subalgebra in $\pg_\Beta$. We set
\[
	\Hb := \Kb \exp(\hg_\Beta), \qquad \Slb := \Hb \Ub, \qquad \Lie(\Slb):= \slg_\Beta.
\]
At Lie algebra level we have an orthogonal decomposition $\qgb = \RR \Beta \oplus \slgb$.

One of the most important properties $\Qb$ has is that 
\[
	\Gl(\ggo) = \Or(\ggo) \Qb.
\]
If in addition $\Beta$ preserves the decomposition $\ggo = \hg \oplus \mg$, then by \cite[(46)]{BL17} and Appendix \ref{app_homogeneous} the subgroup 
\begin{eqnarray}\label{eqn_QbHm}
  	\Q^\Hh_\Beta := \Qb \cap \GHm  \, \subset \, \Gm
\end{eqnarray}
acts transitively on the space of $\Ad(\Hh)$-invariant scalar products on $\mg$. The latter is of course in one-to-one correspondence with the space $\mca^\G(\G/\Hh)$ of $\G$-invariant metrics. Thus,
\begin{equation}\label{eqn_QbHtrans}
	\mca^\G(\G/\Hh) = \Qb^\Hh \cdot \bar g.
\end{equation}

\begin{theorem}\cite{GIT,BL17}\label{thm_beta}
Given the homogeneous space $\G/\Hh$, there exists a $\bar g$-symmetric endomorphism $\Beta \in \End(\ggo)$ associated to it, normalized so that $\tr \Beta = -1$, and satisfying:
\begin{enumerate}[1.]
	\item \emph{(Positivity)} The endomorphism $\Beta^+ := \Beta / \Vert \Beta \Vert^2 + \Id_\ggo$ is positive semi-definite, and its kernel is the orthogonal complement of the nilradical of $\ggo$. In particular, $\Beta^+$ preserves $\mg$. 
	\item \emph{(Automorphism group constraint)} We have that $\Aut(\ggo) \leq \Slb$.
	\item \emph{(Ricci curvature estimate)} For any $\G$-invariant metric $g = q \cdot \bar g$ on $\G/\Hh$, $q\in \Qb^\Hh$, its Ricci endomorphism satisfies
	\[
		\tr \Ricci_g q \Beta^+ q^{-1}  + g(\mcv_g, \mcv_g) \geq 0,
	\] 
	with equality if and only if $q \Beta^+ q^{-1} \in \Der(\ggo)$. Here $\mcv_g \in \mg$ denotes the mean curvature vector of $(\G/\Hh, g)$, and by $\Beta^+$ we mean  $\Beta^+|_\mg$, see condition 1.
\end{enumerate}
\end{theorem}

\begin{proof}
Let $\Beta$ be the stratum label corresponding to $\ggo$, see \cite[$\S$5]{BL17}.
Recall that for a symmetric endomorphism, the kernel is the orthogonal complement of the image. Thus, property 1.~follows from Lemma 5.1 in \cite{BL17}. Property 2.~is simply 	Corollary 4.11 in \cite{BL17}. Regarding 3.~, in the notation of \cite{BL17} we write $\Ricci_g = q \Ricci_\mu q^{-1}$, where $\mu = q^{-1} \cdot \lb$ and $\lb$ is the Lie bracket of $\ggo$. We may further decompose
\[
	\Ricci_\mu = \Ricci^*_\mu - S(\ad_\mu \mcv_\mu),
\]
with $\Ricci_\mu^* \perp \Der(\mu)$, and notice that $\mcv_\mu = q^{-1} \mcv_g$. Thus.
\[
	\tr \Ricci_g q \Beta^+ q^{-1}  + g(\mcv_g, \mcv_g) = \tr \Ricci_\mu^* \Beta^+ - \tr (\ad_\mu \mcv_\mu)\Beta^+ + \bar g (\mcv_\mu, \mcv_\mu).
\]
Since $\ad_\mu \mcv_\mu \in \Der(\mu) \subset \slg_\Beta \perp \Beta$, the second term equals $-\tr \ad_\mu \mcv_\mu = -\bar g(\mcv_\mu, \mcv_\mu)$ by definition of $\mcv_\mu$. Thus, the right-hand-side equals $\tr \Ricci_\mu^* \Beta^+$, which is non-negative by Lemma 6.2 in \cite{BL17}. Finally, the equality condition follows from $\Der(\mu) = q^{-1} \Der(\ggo) q$.
\end{proof}



\begin{lemma}\label{lem_LQbeta}
Let $S \in \End(\ggo)$ be a $\bar g$-symmetric endomorphism, and $Q \in \qgb$ be such that $S + Q \in \sog(\ggo, \bar g)$. Then, $\tr S [Q, \Beta] \geq 0$, with equality if and only if $[Q, \Beta] = 0$.
\end{lemma}

\begin{proof}
 Since the map $E \mapsto -E^T$ is an involutive automorphism of the Lie algebra $\End(\ggo)$, the $\ad(\Beta)$-eigenspaces introduced above satisfy in addition $\End(\ggo)_\lambda^T = \End(\ggo)_{-\lambda}$ for all $\lambda \geq 0$. From this, it is clear that an endomorphism
 \[
 	E := E_0 + \sum_{\lambda>0} (E_\lambda + E_{-\lambda}), \qquad E_\lambda \in \End(\ggo)_\lambda \, \, \forall \lambda,
 \] 
 is symmetric if and only if $E_{-\lambda} = E_\lambda^T$ for all $\lambda \geq 0$, and it is skew-symmetric if and only if $E_{-\lambda} = - E_\lambda^T$ for all $\lambda \geq 0$.

By definition of $\qgb$, we may write $Q = \sum_{\lambda \geq 0} Q_\lambda$, with $Q_\lambda \in \End(\ggo)_\lambda$. Since $S$ is symmetric we may write it as $S= S_0 + \sum_{\lambda>0} (S_\lambda + S_\lambda^T)$. In view of the above observation, the condition $S+Q\in \sog(\ggo,\bar g)$ implies $Q_\lambda = - 2 S_\lambda$, for all $\lambda > 0$. Hence,
\[
	\tr S [Q , \Beta] = - \tr[\Beta, Q] S = - \sum_{\lambda >0} \lambda \tr Q_\lambda S = -\sum_{\lambda > 0} \lambda \la Q_\lambda, S \ra = 2 \, \sum_{\lambda>0} \lambda \, \Vert S_\lambda \Vert^2 \geq 0.
\]
Equality occurs if and only if $S_\lambda = 0$ for all $\lambda > 0$, that is, $[S, \Beta] = 0$. Moreover, since $Q_\lambda = - 2 S_\lambda$ for $\lambda >0$, this is also equivalent to $[Q,\Beta] = 0$.
\end{proof}

\end{appendix}

\bibliography{../bib/ramlaf2}

\def\cprime{$'$} \renewcommand{\MR}[1]{} \newcommand{\noop}[1]{}
\providecommand{\bysame}{\leavevmode\hbox to3em{\hrulefill}\thinspace}
\providecommand{\MR}{\relax\ifhmode\unskip\space\fi MR }
\providecommand{\MRhref}[2]{%
  \href{http://www.ams.org/mathscinet-getitem?mr=#1}{#2}
}
\providecommand{\href}[2]{#2}
\begin{thebibliography}{CDS15b}

\bibitem[AK75]{AlkKml}
Dmitri Alekseevski{\u\i} and Boris~N. Kimel{\cprime}fel{\cprime}d,
  \emph{Structure of homogeneous {R}iemannian spaces with zero {R}icci
  curvature}, Funktional. Anal. i Prilov Zen. \textbf{9} (1975), no.~2, 5--11.

\bibitem[AKL89]{AKL}
Michael~T. Anderson, Peter~B. Kronheimer, and Claude LeBrun, \emph{Complete
  {R}icci-flat {K}\"{a}hler manifolds of infinite topological type}, Comm.
  Math. Phys. \textbf{125} (1989), no.~4, 637--642. \MR{1024931}

\bibitem[AL15]{semialglow}
Romina~M. Arroyo and Ramiro Lafuente, \emph{Homogeneous {R}icci solitons in low
  dimensions}, Int. Math. Res. Not. IMRN (2015), no.~13, 4901--4932.
  \MR{3439095}

\bibitem[AL17]{AL16}
Romina~M. Arroyo and Ramiro~A. Lafuente, \emph{The {A}lekseevskii conjecture in
  low dimensions}, Math. Ann. \textbf{367} (2017), no.~1-2, 283--309.
  \MR{3606442}

\bibitem[Ale75]{Alek75}
Dmitri Alekseevski{\u\i}, \emph{Homogeneous {R}iemannian spaces of negative
  curvature}, Mat. Sb. \textbf{96} (1975), 93--117.

\bibitem[And10]{And2010}
Michael~T. Anderson, \emph{A survey of {E}instein metrics on 4-manifolds},
  Handbook of geometric analysis, {N}o. 3, Adv. Lect. Math. (ALM), vol.~14,
  Int. Press, Somerville, MA, 2010, pp.~1--39. \MR{2743446}

\bibitem[Aub76]{Aub}
Thierry Aubin, \emph{\'equations du type {M}onge-{A}mp\`ere sur les
  vari\'et\'es k\"ahleriennes compactes}, C. R. Acad. Sci. Paris S\'er. A-B
  \textbf{283} (1976), no.~3, Aiii, A119--A121. \MR{0433520}

\bibitem[BB78]{BB}
Lionel B{\'e}rard-Bergery, \emph{Sur la courbure des m\'etriques riemanniennes
  invariantes des groupes de {L}ie et des espaces homog\`enes}, Ann. Sci.
  \'Ecole Norm. Sup. (4) \textbf{11} (1978), no.~4, 543--576.

\bibitem[Bes87]{Bss}
Arthur~L. Besse, \emph{Einstein manifolds}, Ergebnisse der Mathematik und ihrer
  Grenzgebiete (3) [Results in Mathematics and Related Areas (3)], vol.~10,
  Springer-Verlag, Berlin, 1987.

\bibitem[BGK05]{BGK}
Charles~P. Boyer, Krzysztof Galicki, and J{\'a}nos Koll{\'a}r, \emph{Einstein
  metrics on spheres}, Ann. of Math. (2) \textbf{162} (2005), no.~1, 557--580.

\bibitem[Bia98]{Bian1898}
Luigi Bianchi, \emph{Sugli spazi a tre dimensioni che ammettono un gruppo
  continuo di movimenti}, Mem. Soc. Ital. delle Scienze (3) \textbf{11} (1898),
  267--352.

\bibitem[Biq13]{Biq2013}
Olivier Biquard, \emph{D\'{e}singularisation de m\'{e}triques d'{E}instein.
  {I}}, Invent. Math. \textbf{192} (2013), no.~1, 197--252. \MR{3032330}

\bibitem[Biq16]{Biq2017}
\bysame, \emph{D\'{e}singularisation de m\'{e}triques d'{E}instein. {II}},
  Invent. Math. \textbf{204} (2016), no.~2, 473--504. \MR{3489703}

\bibitem[BL17]{GIT}
Christoph B{\"o}hm and Ramiro~A. Lafuente, \emph{Real geometric invariant
  theory}, preprint (ar{X}iv:1701.00643v3), 2017.

\bibitem[BL18]{BL17}
\bysame, \emph{Immortal homogeneous {R}icci flows}, Invent. Math. \textbf{212}
  (2018), no.~2, 461--529. \MR{3787832}

\bibitem[B{\"o}h98]{Bohm98}
Christoph B{\"o}hm, \emph{Inhomogeneous {E}instein metrics on low-dimensional
  spheres and other low-dimensional spaces}, Invent. Math. \textbf{134} (1998),
  no.~1, 145--176. \MR{1646591}

\bibitem[B{\"o}h99a]{Bhm1999}
\bysame, \emph{Non-compact cohomogeneity one {E}instein manifolds}, Bull. Soc.
  Math. France \textbf{127} (1999), no.~1, 135--177. \MR{1700472}

\bibitem[B{\"o}h99b]{Bhm99}
\bysame, \emph{Non-existence of cohomogeneity one {E}instein metrics}, Math.
  Ann. \textbf{314} (1999), no.~1, 109--125. \MR{1689265}

\bibitem[B{\"o}h05]{Bhm05}
\bysame, \emph{Non-existence of homogeneous {E}instein metrics}, Comment. Math.
  Helv. \textbf{80} (2005), no.~1, 123--146.

\bibitem[Bry87]{Bry}
Robert~L. Bryant, \emph{Metrics with exceptional holonomy}, Ann. of Math. (2)
  \textbf{126} (1987), no.~3, 525--576. \MR{916718}

\bibitem[BS89]{BrSa}
Robert~L. Bryant and Simon~M. Salamon, \emph{On the construction of some
  complete metrics with exceptional holonomy}, Duke Math. J. \textbf{58}
  (1989), no.~3, 829--850. \MR{1016448}

\bibitem[BWZ04]{BWZ}
C.~B{\"o}hm, M.~Wang, and W.~Ziller, \emph{A variational approach for compact
  homogeneous {E}instein manifolds}, Geom. Funct. Anal. \textbf{14} (2004),
  no.~4, 681--733.

\bibitem[CDS15a]{CDS1}
Xiuxiong Chen, Simon Donaldson, and Song Sun, \emph{K\"ahler-{E}instein metrics
  on {F}ano manifolds. {I}: {A}pproximation of metrics with cone
  singularities}, J. Amer. Math. Soc. \textbf{28} (2015), no.~1, 183--197.
  \MR{3264766}

\bibitem[CDS15b]{CDS2}
\bysame, \emph{K\"ahler-{E}instein metrics on {F}ano manifolds. {II}: {L}imits
  with cone angle less than {$2\pi$}}, J. Amer. Math. Soc. \textbf{28} (2015),
  no.~1, 199--234. \MR{3264767}

\bibitem[CDS15c]{CDS3}
\bysame, \emph{K\"ahler-{E}instein metrics on {F}ano manifolds. {III}: {L}imits
  as cone angle approaches {$2\pi$} and completion of the main proof}, J. Amer.
  Math. Soc. \textbf{28} (2015), no.~1, 235--278. \MR{3264768}

\bibitem[DM88]{Dtt88}
Isabel Dotti~Miatello, \emph{Transitive group actions and ricci curvature
  properties.}, Michigan Math. J \textbf{35} (1988), no.~3, 427--434.

\bibitem[EW00]{EscWng00}
J.-H. Eschenburg and McKenzie~Y. Wang, \emph{The initial value problem for
  cohomogeneity one {E}instein metrics}, J. Geom. Anal. \textbf{10} (2000),
  no.~1, 109--137. \MR{1758585}

\bibitem[FH17]{FH2017}
Lorenzo Foscolo and Mark Haskins, \emph{New {$G_2$}-holonomy cones and exotic
  nearly {K}\"{a}hler structures on {$S^6$} and {$S^3\times S^3$}}, Ann. of
  Math. (2) \textbf{185} (2017), no.~1, 59--130. \MR{3583352}

\bibitem[GJ17]{GJ15}
Carolyn~S Gordon and Michael~R Jablonski, \emph{Einstein solvmanifolds have
  maximal symmetry}, J. Differential Geom. (to appear) (2017).

\bibitem[GW88]{GrdWls}
Carolyn~S. Gordon and Edward~N. Wilson, \emph{Isometry groups of {R}iemannian
  solvmanifolds}, Trans. Amer. Math. Soc. \textbf{307} (1988), no.~1, 245--269.

\bibitem[GWZ08]{GWZ}
Karsten Grove, Burkhard Wilking, and Wolfgang Ziller, \emph{Positively curved
  cohomogeneity one manifolds and 3-{S}asakian geometry}, J. Differential Geom.
  \textbf{78} (2008), no.~1, 33--111. \MR{2406265}

\bibitem[Heb98]{Heb}
Jens Heber, \emph{Noncompact homogeneous {E}instein spaces}, Invent. Math.
  \textbf{133} (1998), no.~2, 279--352.

\bibitem[Hel01]{Helgason01}
Sigurdur Helgason, \emph{Differential geometry, {L}ie groups, and symmetric
  spaces}, Graduate Studies in Mathematics, vol.~34, American Mathematical
  Society, Providence, RI, 2001, Corrected reprint of the 1978 original.
  \MR{1834454 (2002b:53081)}

\bibitem[Hit74]{Hit1974}
Nigel Hitchin, \emph{Compact four-dimensional {E}instein manifolds}, J.
  Differential Geometry \textbf{9} (1974), 435--441. \MR{0350657}

\bibitem[HN11]{HlgNeb}
Joachim Hilgert and Karl-Hermann Neeb, \emph{Structure and geometry of {L}ie
  groups}, Springer Monographs in Mathematics, vol. 110, Springer, 2011.

\bibitem[HPW15]{HePtrWyl}
Chenxu He, Peter Petersen, and William Wylie, \emph{Warped product {E}instein
  metrics on homogeneous spaces and homogeneous {R}icci solitons}, J. Reine
  Angew. Math. \textbf{707} (2015), 217--245. \MR{3403459}

\bibitem[Jab14]{Jbl13b}
Michael Jablonski, \emph{Homogeneous {R}icci solitons are algebraic}, Geom.
  Topol. \textbf{18} (2014), no.~4, 2477--2486. \MR{3268781}

\bibitem[Jab15a]{Jbl2015}
\bysame, \emph{Homogeneous {R}icci solitons}, J. Reine Angew. Math.
  \textbf{699} (2015), 159--182. \MR{3305924}

\bibitem[Jab15b]{Jab15}
\bysame, \emph{Strongly solvable spaces}, Duke Math. J. \textbf{164} (2015),
  no.~2, 361--402. \MR{3306558}

\bibitem[Joy96a]{Joy1}
D.~D. Joyce, \emph{Compact {$8$}-manifolds with holonomy {${\rm Spin}(7)$}},
  Invent. Math. \textbf{123} (1996), no.~3, 507--552. \MR{1383960}

\bibitem[Joy96b]{Joy2}
Dominic~D. Joyce, \emph{Compact {R}iemannian {$7$}-manifolds with holonomy
  {$G_2$}. {I}, {II}}, J. Differential Geom. \textbf{43} (1996), no.~2,
  291--328, 329--375. \MR{1424428}

\bibitem[Joy00]{Joybook}
\bysame, \emph{Compact manifolds with special holonomy}, Oxford Mathematical
  Monographs, Oxford University Press, Oxford, 2000. \MR{1787733}

\bibitem[JP17]{JblPet14}
Michael Jablonski and Peter Petersen, \emph{A step towards the {A}lekseevskii
  conjecture}, Math. Ann. \textbf{368} (2017), no.~1-2, 197--212. \MR{3651571}

\bibitem[Kna02]{Knapp2e}
Anthony~W. Knapp, \emph{Lie groups beyond an introduction}, second ed.,
  Progress in Mathematics, vol. 140, Birkh\"auser Boston, Inc., Boston, MA,
  2002. \MR{1920389 (2003c:22001)}

\bibitem[Kos65]{Kos65}
J.-L. Koszul, \emph{Lectures on groups of transformations}, Notes by R. R.
  Simha and R. Sridharan. Tata Institute of Fundamental Research Lectures on
  Mathematics, No. 32, Tata Institute of Fundamental Research, Bombay, 1965.
  \MR{0218485}

\bibitem[Laf15]{scalar}
Ramiro~A. Lafuente, \emph{Scalar {C}urvature {B}ehavior of {H}omogeneous
  {R}icci {F}lows}, J. Geom. Anal. \textbf{25} (2015), no.~4, 2313--2322.
  \MR{3427126}

\bibitem[Lau09]{cruzchica}
Jorge Lauret, \emph{Einstein solvmanifolds and nilsolitons}, New developments
  in {L}ie theory and geometry, Contemp. Math., vol. 491, Amer. Math. Soc.,
  2009, pp.~1--35.

\bibitem[Lau10]{standard}
\bysame, \emph{Einstein solvmanifolds are standard}, Ann. of Math. (2)
  \textbf{172} (2010), no.~3, 1859--1877.

\bibitem[Lau11]{solvsolitons}
\bysame, \emph{Ricci soliton solvmanifolds}, J. Reine Angew. Math. \textbf{650}
  (2011), 1--21.

\bibitem[LeB95]{Leb1}
Claude LeBrun, \emph{Einstein metrics and {M}ostow rigidity}, Math. Res. Lett.
  \textbf{2} (1995), no.~1, 1--8. \MR{1312972}

\bibitem[LeB96]{Leb2}
\bysame, \emph{Four-manifolds without {E}instein metrics}, Math. Res. Lett.
  \textbf{3} (1996), no.~2, 133--147. \MR{1386835}

\bibitem[LL14]{alek}
Ramiro Lafuente and Jorge Lauret, \emph{Structure of homogeneous {R}icci
  solitons and the {A}lekseevskii conjecture}, J. Differential Geom.
  \textbf{98} (2014), no.~2, 315--347. \MR{3263520}

\bibitem[Mil76]{Mln}
John Milnor, \emph{Curvatures of left-invariant metrics on {L}ie groups}, Adv.
  Math. \textbf{21} (1976), 293--329.

\bibitem[Nab10]{Nab10}
Aaron Naber, \emph{Noncompact shrinking four solitons with nonnegative
  curvature}, J. Reine Angew. Math. \textbf{645} (2010), 125--153.

\bibitem[Nik00]{Nkn2}
Yu.~G. Nikonorov, \emph{On the {R}icci curvature of homogeneous metrics on
  noncompact homogeneous spaces}, Sibirsk. Mat. Zh. \textbf{41} (2000), no.~2,
  421--429, iv.

\bibitem[Pal61]{Pal61}
Richard~S. Palais, \emph{On the existence of slices for actions of non-compact
  {L}ie groups}, Ann. of Math. (2) \textbf{73} (1961), 295--323. \MR{0126506}

\bibitem[PW09]{PW}
Peter Petersen and William Wylie, \emph{On gradient {R}icci solitons with
  symmetry}, Proc. Amer. Math. Soc. \textbf{137} (2009), 2085--2092.

\bibitem[Sto92]{Sto}
Stephan Stolz, \emph{Simply connected manifolds of positive scalar curvature},
  Ann. of Math. (2) \textbf{136} (1992), no.~3, 511--540. \MR{1189863}

\bibitem[Tho69]{Tho1969}
John~A. Thorpe, \emph{Some remarks on the {G}auss-{B}onnet integral}, J. Math.
  Mech. \textbf{18} (1969), 779--786. \MR{0256307}

\bibitem[Tia15]{Tian2015}
Gang Tian, \emph{K-stability and {K}\"ahler-{E}instein metrics}, Comm. Pure
  Appl. Math. \textbf{68} (2015), no.~7, 1085--1156. \MR{3352459}

\bibitem[Var84]{Varad84}
V.~S. Varadarajan, \emph{Lie groups, {L}ie algebras, and their
  representations}, Graduate Texts in Mathematics, vol. 102, Springer-Verlag,
  New York, 1984, Reprint of the 1974 edition. \MR{746308}

\bibitem[Wan99]{Wang0}
McKenzie~Y. Wang, \emph{Einstein metrics from symmetry and bundle
  constructions}, Surveys in differential geometry: essays on {E}instein
  manifolds, Surv. Differ. Geom., vol.~6, Int. Press, Boston, MA, 1999,
  pp.~287--325. \MR{1798614}

\bibitem[Wan12]{Wang}
McKenzie Y.-K. Wang, \emph{Einstein metrics from symmetry and bundle
  constructions: a sequel}, Differential geometry, Adv. Lect. Math. (ALM),
  vol.~22, Int. Press, Somerville, MA, 2012, pp.~253--309. \MR{3076055}

\bibitem[WZ86]{WZ86}
McKenzie~Y. Wang and Wolfgang Ziller, \emph{Existence and nonexistence of
  homogeneous {E}instein metrics}, Invent. Math. \textbf{84} (1986), no.~1,
  177--194.

\bibitem[Yau78]{Yau}
Shing~Tung Yau, \emph{On the {R}icci curvature of a compact {K}\"ahler manifold
  and the complex {M}onge-{A}mp\`ere equation. {I}}, Comm. Pure Appl. Math.
  \textbf{31} (1978), no.~3, 339--411. \MR{480350}

\end{thebibliography}
\bibliographystyle{amsalpha}

\end{document}